\documentclass[11pt]{amsart}
\usepackage{graphicx}
\usepackage{amssymb}
\usepackage{amsmath}
\usepackage{amscd}
\usepackage{amsthm,amsfonts,amsxtra,amssymb,amscd}
\usepackage[all]{xy}
\usepackage{enumerate}
\theoremstyle{plain}
\newtheorem*{lemma*}{Lemma}
\newtheorem{lemma}[subsection]{Lemma}
\newtheorem*{theorem*}{Theorem}
\newtheorem{theorem}[subsection]{Theorem}
\newtheorem*{proposition*}{Proposition}
\newtheorem{proposition}[subsection]{Proposition}
\newtheorem*{corollary*}{Corollary}
\newtheorem{corollary}[subsection]{Corollary}
\theoremstyle{definition}
\newtheorem*{definition*}{Definition}
\newtheorem{definition}[subsection]{Definition}
\newtheorem*{example*}{Example}
\newtheorem{example}[subsection]{Example}
\theoremstyle{remark}
\newtheorem*{remark*}{Remark}
\newtheorem{remark}[subsection]{Remark}

\newcommand{\R}{{\mathbb R}}

\newcommand{\CV}{{\mathcal{V}}}

\newcommand{\C}{{\mathbb C}}
\newcommand{\Z}{{\mathbb Z}}
\renewcommand{\ll}{{\langle}}
\newcommand{\rr}{{\rangle}}

 \newcommand{\y}{{{\bf y}}}

 \newcommand{\p}{{{\bf p}}}

\newcommand{\CH}{{\mathcal{H}}}
\newcommand{\CB}{{\mathcal{B}}}

\newcommand{\CT}{{\mathcal{T}}}

\newcommand{\CS}{{\mathcal{S}}}

 \renewcommand{\c}{{\mathfrak{c}}}
\newcommand{\s}{{\bf{s}}}

\renewcommand{\r}{{\bf{r}}}

\title[Poisson summation and box spline]{
 Poisson summation  formula  and Box splines}
\author{ Mich{\`e}le Vergne}
\address{Universit\'e Denis-Diderot-Paris 7, Institut de Math\'ematiques de Jussieu,
 C.P.~7012\\ 2~place Jussieu,   F-75251 Paris~cedex~05}
\email{vergne@math.jussieu.fr}

\begin{document}

\maketitle
 \tableofcontents

\section{Introduction}

Let $V$ be a  finite dimensional real vector space, of dimension $d$,  equipped with a lattice $\Lambda$. Let $U$ be the dual vector space of $V$, and $\Gamma\subset U$ the dual lattice of $\Lambda$.
For $v\in V$, denote by $\partial_v$ the differentiation in the direction $v$. We denote by $\delta_\lambda$ the Dirac measure at the point $\lambda\in V$.

Let $\Phi:=[\alpha_1,\alpha_2,\ldots, \alpha_N]$ be a list of elements in $\Lambda$.
The  box spline $B(\Phi)$ is the measure on $V$ such that, for a  continuous function  $F$ on $V$,
$$
\langle B(\Phi),F\rangle=
\int_{0}^{1}\cdots \int_{0}^{1}
F(\sum_{k=1}^{N} t_k\alpha_k) dt_1\cdots dt_N.
$$

The support of $B(\Phi)$ is the zonotope $Z(\Phi):=
\{\sum_{k=1}^{N} t_k\alpha_k; 0\leq t_k\leq 1\}$.
If $\Phi$ generates $V$, then the measure $B(\Phi)$ is given by integration against a  piecewise polynomial function, that we  denote by $b(\Phi)$.

The Fourier transform  ${\hat B}(\Phi)(x)$  is the  analytic function of $x\in U$:
$${\hat B}(\Phi)(x)=\prod_{\alpha\in \Phi}\frac{e^{i\langle \alpha,x\rangle}-1}{i\langle \alpha,x\rangle}.$$
Remark that the inverse of the Box spline is related to the generating function for Todd classes. We thus denote it by
$$Todd(\Phi,x)=\prod_{\alpha\in \Phi}\frac{i\langle \alpha,x\rangle}{e^{i\langle \alpha,x\rangle}-1}.$$
It is only defined for $x$ small enough.

 Denote by ${\mathcal C}(\Lambda)$ the space of complex valued functions on $\Lambda$.
If  $m\in {\mathcal C}(\Lambda)$,
 let $$b(m)=\sum_{\lambda\in \Lambda} m(\lambda)\delta_\lambda* b(\Phi)$$
the convolution of  the discrete measure $\sum_{\lambda\in \Lambda} m(\lambda)\delta_\lambda$ with $b(\Phi)$.
Thus $b(m)$  is a locally polynomial measure on $V$.

Recall that the list $\Phi$ is called  unimodular if any basis of $V$ contained in $\Phi$  is a basis of the lattice $\Lambda$.
For simplicity, we restrict to the unimodular case in this introduction.
In this case, Dahmen-Micchelli \cite{DM1} proved that  the convolution $m\mapsto b(m)$ is injective. Let us recall  Dahmen-Micchelli formula for the inverse map.
By Fourier transform, convolution becomes  the multiplication by $\hat B(\Phi)$, and  it is thus tempting to use Fourier transform to invert the convolution.
Indeed we obtain (in case where $m$ is compactly supported)
\begin{equation}\label{eq:todd}
\sum_{\lambda}m(\lambda)e^{i\ll\lambda,x\rr}=\int_V e^{i\ll v,x\rr} Todd(\Phi,x) b(m)(v)dv
\end{equation}
 for $x$ small.

Replace
  $Todd(\Phi,x)$ by its Taylor series and consider the Todd operator
$$Todd(\Phi,\partial)=\prod_{\alpha\in \Phi}\frac{\partial_\alpha}{1-e^{-\partial_\alpha}} =1+\frac{1}{2}\sum_{\alpha\in \Phi}\partial_\alpha+\cdots ,$$
an infinite series of constant coefficients differential operators on $V$. Thus, for $x$ small, after integration by parts, we still  have
\begin{equation}\label{eq:toddbis}
\sum_{\lambda\in \Lambda}m(\lambda)e^{i\ll\lambda,x\rr}=\int_V (Todd(\Phi,\partial) b(m))(v) e^{i\ll v,x\rr} dv.
\end{equation}

Miraculously, this equation still holds if we replace the integral on $V$ be the summation in  $\Lambda$, in an appropriate limit sense. Indeed we have the identity for all $x\in U$,
\begin{equation}\label{eq:toddter}
\sum_{\lambda\in \Lambda}m(\lambda)e^{i\ll\lambda,x\rr}=\lim_{t>0, t\to 0}\sum_{\lambda\in \Lambda} (Todd(\Phi,\partial)b(m))(\lambda+t\epsilon)e^{i\ll \lambda,x\rr}.
\end{equation}
Here    $\epsilon$ is a generic vector in the cone generated by the
$\alpha_k$, and $\lambda+t\epsilon$ is a regular point in $V$ (the notion of generic and regular vectors is defined in the article).
As $b(m)$ is a piecewise polynomial function,  only a finite number of terms  in the series  $Todd(\Phi,\partial)b(m)$  do not vanish at the regular points $\lambda+t\epsilon$, so that the formula is well defined.
Thus Dahmen-Micchelli deconvolution formula is:
\begin{equation}\label{eq:deconv}
m(\lambda)=\lim_{t>0,t\to 0} (Todd(\Phi,\partial)b(m))(\lambda+t\epsilon).
\end{equation}

In this article, we prove a slightly more general formula.

Let ${\bf y}=[y_1,y_2,\ldots,y_N]$ be a list of $N$ complex numbers.
The box spline $B(\Phi,\y)$, with parameter $\y$,  is the measure on $V$ such that, for a  continuous function  $F$ on $V$,
$$
\langle B(\Phi,\y),F\rangle=
\int_{0}^{1}\cdots \int_{0}^{1} e^{i\sum_{k=1}^N t_k y_k}
F(\sum_{k=1}^{N} t_k\alpha_k) dt_1\cdots dt_N.
$$
The Fourier transform  ${\hat B}(\Phi,\y)(x)$  is the function
 $${\hat B}(\Phi,\y)(x)=\prod_{k=1}^N\frac{e^{i(\langle \alpha_k,x\rangle+y_k)}-1}{i(\langle \alpha_k,x\rangle+y_k)}.$$

If  $m\in {\mathcal C}(\Lambda)$,
let
$$b(\y,m)=\sum_{\lambda\in \Lambda}m(\lambda) \delta_\lambda*  B(\Phi,\y).$$

We also  consider centered Box splines, defined using convolution of intervals $\{t\alpha_k,-\frac{1}{2}\leq t\leq \frac{1}{2}\}$,  and more generally translation of the Box spline by a parameter $r$ in the zonotope.
 We prove a   deconvolution formula similar to (\ref{eq:deconv}) for the translated box spline with parameters.
We show that the deconvolution formula  allows us to recover $m$ from  $b(\y,m)$ by  an uniform formula on
all points of $\Lambda\cap \Delta$, where $\Delta$ is a generic translation of the zonotope.
In particular, an interesting case is when $\y=0$. If the function $b(0,m)$ is polynomial on a domain $\Omega$, we obtain that $m$ is polynomial on the enlarged domain $\Omega-Z(\Phi)$.
Interesting examples of this phenomenon occur in the case of the Kostant partition function. Indeed in this case the convolution of the partition function with the Box spline  is simply a convolution of Heaviside functions, with domains of polynomiality given by the so called big chambers. More generally, these examples occur in Hamiltonian geometry, where $b(0,m)$ is the Duistermaat-Heckman measure and is polynomial on each connected component of the set of regular values of the moment map.
For example, we have used the deconvolution formula in \cite{DV} to study qualitative properties of some branching rules, for reductive noncompact Lie groups, even in the absence of explicit character formulae.

\bigskip

Let us comment on the technique used in this article.

The problem of inverting the convolution with the Box spline is  equivalent to the problem of describing the  function  number of integral points in polytopes in terms of volumes, and we could have applied results of  \cite{SZ}, \cite{BV}.
We also  gave a proof of the deconvolution formula in \cite{dpv1} for $\y=0$,  based on a detailed study of the Dahmen-Micchelli spaces of functions on $\Lambda$.

A  method by Poisson formula   was used in an unpublished article with Michel Brion to obtain formulae for partition functions, as an alternate method to the  cone decomposition method of \cite{BV}. Here we use a mixed method between  \cite{SZ}, \cite{BV}. We use a crucial lemma of \cite{SZ}, and  we follow  several of the steps of the unpublished article with Brion. So there is no new idea in this article.
However our inversion formula is slightly more general, and we believe we have clarified some of the delicate points. In particular we describe more precisely the regions of quasi polynomial behavior of $m$ in terms of the regions of polynomiality of $b(m)$.
 Furthermore, it is   stated in a quite natural  way ( we state the inversion formula  in such  a way that it is impossible to make signs mistakes, for example), and we believe that our Poisson  method is straightforward.
However, as it should be, in all these methods, the  limiting procedures are delicate. Indeed the deconvolution formula itself  is delicate.

\bigskip

Let us sketch the proof of Formula (\ref{eq:toddter})
 in the case where $m$ is the delta function at $0$  of the lattice $\Lambda$, and $\Phi$ is unimodular.
We thus need to prove the identity
$$1=\lim_{t>0, t\to 0}\left(\sum_{\lambda\in \Lambda}(Todd_{[L]}(\Phi,\partial)b(\Phi))(\lambda+t\epsilon)e^{i\ll \lambda,x\rr}\right)$$
where
$Todd_{[L]}(\Phi,\partial)$ is the series $Todd(\Phi,\partial)$ truncated at some sufficiently large order $L$, the higher  terms giving a zero contribution on the regular element $\lambda+t\epsilon$  of $V$.

We compute the term in parenthesis using Poisson formula.

Define $F([q],x)=\sum_{a=0}^{\infty} q^a F_a(x)$
 to be  the Taylor series
  of $$
  Todd(\Phi,qx)\hat B(\Phi)(x)=\prod_{k=1}^Nq(\frac{ e^{i\ll\alpha_k,x\rr}-1}{e^{iq\ll\alpha_k,x\rr}-1})$$
 at $q=0$.
%We write also $$Z([q],x)=T([q],x)\prod_{k=1}^N (e^{i\ll\alpha,x\rr}-1), $$
%where $T([q],x)$ is the Taylor series at $q=0$ of
%$\prod_{k=1}^N\frac{q}{e^{iq\ll\alpha,x\rr}-1}$, that is a series of rational functions of $x$.
The Fourier transform of
$Todd_{[L]}(\Phi,\partial)b(\Phi)$
is  the truncated series
$F(x)=\sum_{a=0}^{L}  F_a(x)$, thus
 formally
$$\sum_{\lambda\in \Lambda}Todd_{[L]}(\Phi,\partial)b(\Phi)(\lambda+t\epsilon)e^{i\ll\lambda,x\rr}=
e^{-i\ll t\epsilon,x\rr}\sum_{\gamma\in \Gamma} F(x-2\pi \gamma) e^{i\ll t\epsilon,2\pi \gamma\rr}.$$

 We compute the full formal series
 $$\sum_{\gamma\in \Gamma} F([q],x-2\pi \gamma) e^{i\ll t\epsilon,2\pi \gamma\rr},$$
 this is well defined when $\epsilon$ is generic and $t$ small,  coefficients of $q^a$ obtained by summations over $\Gamma$ are  $0$ on the regular element  $t\epsilon$  when $a$ is sufficiently large, and the limit when $t\to 0$  and $q=1$ is $1$ provided $\epsilon$ is in the cone generated by the $\alpha_k$.
This gives us the wanted result.
In fact, the main result of this article (Theorem \ref{theo;thetalim}) is a generalisation
of the equality of $L^2$-functions of $t\in \R/\Z$ :
$$e^{i\{t\}x}=
\sum_{n\in \Z}\frac{e^{ix}-1}{x-2in}e^{2i\pi nt}.$$
Here $\{t\}=t-[t]\in [0,1]$ is the fractional part of $t$.
As shown on  this example,  limits from right or left of the summation over the lattice are not the same. It is equal to $1$ only if $t$ tends to $0$ from the right.

We need furthermore to  introduce parameters $\y$. In fact,  the use of generic parameters $\y$  simplify the proofs.

Equation (\ref{eq:todd}) is very reminiscent of the ``delocalized" equivariant index formula for an equivariant elliptic operator.
If  $G$ is a torus, we have employed  the deconvolution formula for box splines  in order to obtain
multiplicities formula for the index  of a $G$-equivariant elliptic (or transversally elliptic) operator in terms of  spline functions on the lattice $\Lambda$ of characters   of $G$ (\cite{dpv1}). This slightly more general deconvolution formula proved here is similarly  needed for the proofs of the results announced in \cite{V}.

We thank  Michel Duflo for several comments
 on this manuscript.

\part{ The results}

In this part, we state precisely the theorems proven in this paper.
Thus we start by definitions and notations.

\section{Piecewise analytic functions}

Let $V$ be a finite dimensional real vector space equipped with a lattice $\Lambda$.
Denote by ${\mathcal C}(\Lambda)$ the space of complex valued functions on $\Lambda$. For $\lambda\in \Lambda$,
we denote by $\delta^{\Lambda}_\lambda$ the $\delta$ function
on $\Lambda$ such that $\delta^{\Lambda}_\lambda(\nu)=0$, except for $\nu=\lambda$ where $\delta^{\Lambda}_\lambda(\lambda)=1$.

If $h$ is a distribution on $V$, we denote by $\int_V h(v)f(v)$ its value on a test function $f$.
If $v_0\in V$, we denote by $\delta^{V}_{v_0}$ the $\delta$ distribution on $V$: $\int_V \delta^{V}_{v_0}(v)f(v)=f(v_0)$.

 Choosing  the Lebesgue measure $dv$ determined by $\Lambda$,
  we identify a generalized function $f(v)$ on $V$ to the distribution $f(v)dv$.

 Consider  $\Phi:=[\alpha_1,\alpha_2,\ldots, \alpha_N]$ be a  list of elements in $\Lambda$.
The  box spline $B(\Phi)$ is the measure on $V$ such that, for a  continuous function  $F$ on $V$,
$$
\langle B(\Phi),F\rangle=
\int_{0}^{1}\cdots \int_{0}^{1}
F(\sum_{k=1}^{N} t_k\alpha_k) dt_1\cdots dt_N.
$$

Let $Z(\Phi)=\{z=\sum_{k=1}^{N} t_k\alpha_k; 0\leq t_k\leq 1\}$ be the Minkowski sum of the segments $[0,1]\alpha_i$. The polytope $Z(\Phi)$, called the zonotope, is the support of $B(\Phi)$.

\bigskip

 Assume that $\Phi$ generates $V$.
 An  hyperplane   spanned by elements of $\Phi$ will be called a wall. A translate of a wall by an element of $\Lambda$ will be called an affine wall.

 A point $\epsilon\in V$ is called $\Phi$-generic if $\epsilon$ does not lie on any  wall. A point $v\in V$ is called $\Phi$-regular if $v$ does not lie on any affine wall.
A connected component $\tau$ of the set of $\Phi$-generic elements is called a tope. Thus topes are open cones in $V$.
A connected component ${\mathfrak c}$ of the set of $\Phi$-regular elements is called an alcove.
We denote by $V_{reg}$ the set of $\Phi$-regular elements, that is the disjoint union of the alcoves.
In the rest of this article,  we often just say that $v$ is generic, regular, etc., the system $\Phi$ being implicitly understood.

A piecewise polynomial function $b$ is a function on $V_{reg}$ such that for each alcove  ${\mathfrak c}$, there exists a polynomial function $b^{\mathfrak c}$ on $V$ satisfying $b(v)=b^ {\mathfrak c} (v)$ for $v\in  {\mathfrak c}$ .
A piecewise analytic function
 is a function on $V_{reg}$
 such that for each alcove $ {\mathfrak c}$, there exists an analytic  function $b^ {\mathfrak c}$ on $V$ satisfying $b(v)=b^ {\mathfrak c}(v)$ for $v\in  {\mathfrak c}$.

 We denote by $PW$ the space of piecewise polynomial functions.
We denote by $PW^{\omega}$ the space of piecewise analytic functions.

\begin{definition}\label{def:fc}
If $f\in PW^{\omega}$, and $\c$ is an alcove, we denote by
$f^{\c}$ the analytic function on $V$ coinciding with $f$ on $\c$. \end{definition}

 The lattice $\Lambda$ acts on $PW,PW^{\omega} $ by translation.

If $V=\R$ with lattice $\Lambda=\Z$,    such a piecewise analytic function  admits left and right limits at any point of $V$.
Let us generalize the notion of  left or right limits to our multidimensional  context.

If $v\in V$, and $\epsilon$ is a generic vector, then $v+t\epsilon$ is in $V_{reg}$ if $t>0$ and sufficiently small.
\begin{definition}\label{def:rightlimit}
Let $v\in V$, and $f\in PW$ (or $PW^{\omega})$  .
Let $\epsilon$ be a generic vector.
We define
$(\lim_{\epsilon}f)(v)=\lim_{t>0,t\to 0}f(v+t\epsilon)$.
\end{definition}
Clearly
$(\lim_{\epsilon}f)(v)$ depends only of the tope $\tau$ where $\epsilon$ belongs and is denoted by
$(\lim_{\tau}f)(v)$ in \cite{dpv1}.

  Consider $f\in PW$  (defined on $V_{reg}$) as a locally $L^1$-function on $V$,  thus $f(v)$ defines  a  generalized function on $V$.
 An element of $PW$, considered as a generalized function  on $V$,
will be called a piecewise polynomial generalized function .
 Multiplying by $dv$, we obtain the space of piecewise polynomial distributions on  $V$.
We define similarly piecewise analytic generalized functions
and piecewise analytic distributions on  $V$.

\bigskip

The box spline
is an important example of piecewise polynomial distribution.

Indeed, if  $\Phi$ spans the vector space $V$, $B(\Phi)$ is in the space $PW$ of piecewise polynomial distributions. We will write  $B(\Phi)=b(\Phi)(v)dv$, where
$b(\Phi)(v)$ is a locally polynomial function on $V$.

If $\Psi$ is a sublist of $\Phi$ still spanning $V$, then $B(\Psi)$ and its translates by elements of $\Lambda$   are again in $PW$. In fact alcoves for the system $\Psi$ are larger than alcoves for the system $\Phi$, and $B(\Psi)$ is given by a polynomial function on each alcove for $\Psi$.
If $\Psi$ does not span $V$, the distribution
$B(\Psi)$ vanishes on $V_{reg}$.

Let ${\bf y}=[y_1,y_2,\ldots,y_N]$ be a list of $N$ complex numbers.
The box spline $B(\Phi,\y)$, with parameter $\y$,  is the measure on $V$ such that, for a  continuous function  $F$ on $V$,
$$
\langle B(\Phi,\y),F\rangle=
\int_{0}^{1}\cdots \int_{0}^{1} e^{i(\sum_{k=1}^N t_k y_k)}
F(\sum_{k=1}^{N} t_k\alpha_k) dt_1\cdots dt_N.
$$

 Then, if $\Phi$ spans $V$,  $B(\Phi, \y)=b(\Phi,\y)(v)dv$ where  $b(\Phi,\y)(v)$ is in the space $PW^{\omega}$ of piecewise analytic functions.
In fact, it is a piecewise exponential polynomial function of $v$.

When $\y=[0,0,\ldots,0]$,  $B(\Phi,\y)= B(\Phi)$.

We give examples in dimension $1$.

\begin{example}\label{first}
Let $V=\R \omega$.
We identify $V$ with $\R$: $t\in \R$ is the element $t\omega$ of $V$.

$\bullet$  $\Phi_1=[\omega]$, $\y=[y]$.
Then:

$$b(\Phi_1)(t)=\begin{cases}0\ \  &\text{if\ } t<0 \\ 1 \ \ &\text{if\ } 0<t<1\\
0\ \  &\text{if\ } t>1
\end{cases}$$

while

$$b(\Phi_1,\y)(t)=\begin{cases}0\ \ & \text{if\ } t<0 \\ e^{ity} \ \ &\text{if\ } 0<t<1\\
0\ \ &\text{if\ } t>1
.\end{cases}$$

$\bullet$ Let $\Phi'_2=[\omega,-\omega]$.

Then
$$b(\Phi'_2)(t)=\begin{cases}
0\ \  &\text{if\ } t<-1 \\ \\ t+1 \ \  &\text{if\ } -1<t<0\\ \\
t-1 \ \  &\text{if\ } 0<t<1\\ \\
0\ \  &\text{if\ } t>1 \\
\end{cases}$$
while

$$b(\Phi'_2,\y)(t)=\begin{cases} 0\ \  &\text{if\ } t<-1 \\
\\
\frac{e^{it y_1} e^{i(y_1+y_2)}}{i (  y_1+ y_2) }-
\frac{e^{-it y_2}}{i (  y_1+y_2) }\ \ \  &\text{if\ } -1<t<0\\
 \\
\frac{ e^{i (y_1+y_2)}e^{-it y_2}}{i(y_1+y_2)}- \frac{e^{it y_1}}{i(y_1+y_2)}
 \ \  &\text{if\ } 0<t<1\\
\\
0\ \  &\text{if\ } t>1 \\
\end{cases}$$

$\bullet$  $\Phi=[\omega,2\omega], \y=[y_1,y_2].$

Then
$$B(\Phi)(t)=\begin{cases}0\ \  &\text{if\ } t<0 \\ \\ \frac{t}{2} \ \  &\text{if\ } 0<t<1\\ \\
\frac{1}{2} \ \  &\text{if\ } 1<t<2\\ \\
\frac{(3-t)}{2}\ \  &\text{if\ } 2<t<3 \\ \\
0\ \  &\text{if\ } t>3 \\
\end{cases}$$

while
$$B(\Phi,\y)(t)=\begin{cases} 0\ \  &\text{if\ } t<0 \\
\\
\frac{e^{it y_1}}{i \left( 2 y_1- y_2 \right) }-\frac{e^{it\frac{ y_2}{2}}}{i \left( 2 y_1- y_2 \right) }\ \ \  &\text{if\ } 0<t<1\\
 \\
\frac {e^{it\frac{ y_2}{2}} (e^{i (y_1-\frac{y_2}{2})}-1)}{i \left(2y_1- y_2 \right) }\ \  &\text{if\ } 1<t<2\\
\\
\frac{-e^{it y_1}e^{i (y_2-2y_1)}}{i \left( 2 y_1- y_2 \right) }+\frac{e^{it\frac{ y_2}{2}}e^{i (y_1-\frac{y_2}{2})}}
{i \left( 2 y_1- y_2 \right) }
\ \  &\text{if\ } 2<t<3\\\\
0\ \  &\text{if\ } t>3 \\
\end{cases}$$

\end{example}

In all the examples above, although the explicit formulae for $\CB(\Phi,\y)(t)$ seems to have poles in $\y$, it is easy to verify that on each alcove $\CB(\Phi,\y)(t)$ is an analytic function on $t,\y$.

\begin{definition}
Assume $\Phi$ generates $V$.
We denote by
$\CS$  the space of generalized functions on $V$ generated by the action of constant coefficients differential operators
and  translations by elements of $\Lambda$.
on  the piecewise polynomial $b(\Phi)$.

We denote by
$\CS^{\y}$  the space of generalized functions on $V$ generated by the action of constant coefficients differential operators
 and   translations by elements of $\Lambda$.
 on  the piecewise analytic function $b(\Phi,\y)$
\end{definition}

For example,  $$\prod_{k=1}^N (-\partial_{\alpha_k}+iy_k) b(\Phi,\y)=
\prod_{k=1}^N(e^{iy_k}\delta^{V}_{\alpha_k}-\delta^{V}_0)$$ is in the space $\CS^{\y}.$
 Here the product in the right hand side of this equation is the convolution product, thus
  the right hand side is  a sum of $\delta$-functions on $V$ with coefficients depending on $\y$.

\bigskip

Elements of  $\CS^{\y}$ can be evaluated at any regular point $v\in V_{reg}$. Thus, if $\epsilon$ is generic, we   denote  by
$$\lim_{\epsilon}^{\Lambda}:\CS^{\y}\to {\mathcal C}(\Lambda)$$
the map
 $$(\lim_{\epsilon}^{\Lambda}f)(\lambda)= (\lim_\epsilon f)(\lambda).$$

%The space $\CS(E)$ contains in particular the derivatives of $\delta$ functions at points $\lambda\in \Lambda$.

\bigskip

Let $q$ be a formal variable.
If $E$ is a vector space and $f([q])=\sum_{a=0}^{\infty} q^a f_a$ is a formal series of elements of $E$, we write $f\in E[[q]]$.
If $f(q,x)$ is a smooth function of $q$, defined near $q=0$, and depending of some parameters $x$, we denote by $f([q],x)=\sum_{a=0}^{\infty} q^a f_a(x)$ its Taylor series at $q=0$, a formal series of functions of $x$.
If the series $f([q])$ is finite (or convergent), we write
$f([1])$ or $f([q])|_{q=1}$ for the sum  $\sum_{a=0}^{\infty} f_a$.

Introduce formal series $m([q])=\sum_{a=0}^{\infty} q^a m_a$ of generalized functions  (or of distributions) on $V$. Then if $f$ is a test function
$$\int_{V}m([q])(v)f(v)dv=\sum_{a=0}^{\infty} q^a \int_{V} m_a(v)f(v)dv$$ is a formal power series in $q$.
It may be evaluated at $q=1$ on a test function $f$ if the preceding series is finite (or convergent).
Formal series of distributions occur naturally in the context of Euler-MacLaurin formula.

 If all the elements $m_a$ belong to  $\mathcal S^{\y}$, we write $m([q])\in \CS^{\y}[[q]]$.
 In this case,
$\lim_{\epsilon}^{\Lambda}m([q])=\sum_{a=0}^{\infty} q^a \lim_{\epsilon}^{\Lambda} m_a$ is a formal power series of $q$ with values in $\mathcal C(\Lambda)$. It can be evaluated at $q=1$ if the corresponding series of elements of $\mathcal C(\Lambda)$ is convergent at $q=1$.
  This is for example the case if all, but a finite number, the generalized functions $m_a$ are supported on affine walls.

\begin{definition}
Let $m([q])=\sum_{a=0}^{\infty} q^a m_a$
be a series of generalized functions, with $m_a\in  S^{\y}$.

 Assume that the series $\sum_{a=0}^{\infty} q^{a}\lim_{\epsilon}^{\Lambda} m_a$ is convergent at $q=1$.  We denote
 $$\lim_{\epsilon}^{\Lambda} m([q])|_{q=1}$$
the corresponding element $\sum_{a=0}^{\infty} \lim_{\epsilon}^{\Lambda}m_a$
 of ${\mathcal C}(\Lambda)$.
\end{definition}

We say that
 $m([q])=\sum_{a=0}^{\infty} q^a m_a$  is supported on affine walls if all the elements $m_a$ are supported on affine walls.
 In this case,  $m_a=0$  on $V_{reg}$  for all $a$, and   $$\lim_{\epsilon}^{\Lambda} m([q])|_{q=1}=0.$$

\section{Box splines with parameters}

\subsection{Inversion formula for the box spline. The unimodular case}

Let $U$ be the dual vector space of $V$.
If $b$ is a function on $V$, we denote by $\hat b(x)=\int_V e^{i\ll x,v\rr}b(v)dv$ its Fourier transform.

The Fourier transform
 \begin{equation}\label{eq:fourierbx}
{\hat B}(\Phi)(x)=\int_V e^{i\ll x,v\rr }B(\Phi)(v)
\end{equation}
of the box spline $B(\Phi)$
  is the  analytic function of $U$:
$${\hat B}(\Phi)(x)=\prod_{k=1}^N\frac{e^{i\langle \alpha_k,x\rangle}-1}{i\langle \alpha_k,x\rangle}.$$

Let us first explain the Dahmen-Micchelli inversion formula in the case where $\Phi$ (spanning $V$)  is unimodular, that is any basis of $V$ consisting of elements of  $\Phi$  is a basis of the lattice $\Lambda$.

 Consider the function  $$Todd(\Phi)(x)=\prod_{k=1}^N\frac{i\langle \alpha_k,x\rangle}{e^{i\langle \alpha_k,x\rangle}-1},$$
 that is $Todd(\Phi)(x)$ is
the inverse of ${\hat B}(\Phi)(x)$.  The function  $Todd(\Phi)(x)$ is  defined near $x=0$.
Then, when $x$  is small,
$$Todd(\Phi)(x){\hat B}(\Phi)(x)=1= Todd(\Phi)(x)\int_V e^{i\ll x,v\rr}b(\Phi)(v)dv.$$

It is tempting to  use the fact that Fourier transform exchange  the multiplication by $i\ll \alpha,x \rr$ on functions on $U$ and  the derivation $-\partial_\alpha$ on functions on $V$. We are not allowed to do this as
 $Todd(\Phi)(x)$ is not a polynomial (and is not defined when $x$ is large). Thus
introduce a variable $q$, and consider the Taylor series at $q=0$ of $Todd(\Phi)(qx)$.
$$Todd(\Phi)(qx)=\sum_{a=0}^{\infty}q^a T_a(x)=1-\frac{q}{2}\sum_{k=1}^N i\alpha_k(x)+\cdots .$$
We denote by
$$Todd(\Phi)([q],\partial)=\prod_{k=1}^N \frac{q\partial_{ \alpha_k}}{1-e^{-[q]\partial_{\alpha_k}}}=1+\frac{q}{2} \sum_{k=1}^N \partial_{\alpha_k}+\cdots$$
the corresponding series of differential operators with constant coefficients.
Thus we obtain, for $q$ sufficiently small,
$$Todd(\Phi)(qx)\hat B(\Phi)(x)=\int_V e^{i\ll x,v\rr}
(Todd(\Phi)([q],\partial)b(\Phi))(v)dv.$$

In a certain sense, this equation still holds for $q=1$, provided we replace the integral over $V$ by the sum over $\Lambda$.
Indeed,
$Todd(\Phi)([q],\partial)b(\Phi)|_{q=1}$ restricted to $\Lambda$ in the sense explained below is equal to the Dirac function
$\delta^{\Lambda}_0$ on $\Lambda$, and this clearly satisfies
$$1=\sum_{\lambda\in \Lambda} (Todd(\Phi)([q],\partial)b(\Phi))(\lambda)e^{i\ll x,\lambda\rr}|_{q=1}=Todd(\Phi)(qx)\hat B(\Phi)(x)|_{q=1}.$$

Let us explain now precisely the  results obtained in \cite{dpv1} (for which we will give another  proof in this article).

Consider the series $Todd(\Phi)([q],\partial)b(\Phi)$
of generalized functions on $V$. As $b(\Phi)$ is piecewise polynomial,
this is a series of generalized functions $\sum_{a=0}^{\infty} q^a m_a$, where all the $m_a$ are in $\mathcal S$ and all, but a finite number,  generalized functions $m_a$ are supported on affine walls.

Dahmen-Miccheli theorem is:
\begin{theorem}
If $\epsilon$ is a generic vector {\bf belonging to the cone} $Cone(\Phi)$ generated by the elements of $\Phi$, then
$$\lim_{\epsilon}^{\Lambda}Todd(\Phi)([q],\partial)b(\Phi)|_{q=1}=\delta^{\Lambda}_0.$$
  \end{theorem}

Note that we do not assume that $\Phi$ generates
a salient cone.

%We can rephrase this theorem as the identity of functions of $x\in U$:
%If $\epsilon$ is a generic vector {\bf belonging to the cone} $Cone(\Phi)$ generated by the elements of $\Phi$, then
%$$\sum_{\lambda\in \Lambda}\lim_{\epsilon}^{\Lambda}(Todd(\Phi)([q],\partial)b(\Phi))(\lambda) e^{i\ll x,\lambda\rr}|_{q=1}=1.$$
%

 Let us rephrase this theorem in order that it becomes easy to remember (and to prove and generalize).

 Consider the analytic function
$$F(q,x)=Todd(\Phi)(qx)\hat B(\Phi)(x)$$
that is
$$F(q,x)=q^{N}
\left(\prod_{k=1}^N\frac{e^{i\langle \alpha_k,x\rangle}-1}{e^{iq\langle \alpha_k,x\rangle}-1}\right).$$

For $q=1$, this function is identically equal to $1$.
The power $q^N$ is such that $F(q,x)$ has no pole at $q=0$.

The Taylor series $F([q],x)$ at $q=0$ of $F(q,x)$ is a series of analytic functions of $x$. We have:

 $$F([q],x)= {\hat B}(\Phi)(x)- \frac{q}{2} (\sum_{k=1}^{N}i\alpha_k(x)) {\hat B}(\Phi)(x)+\cdots .$$

\begin{theorem}

Assume that $\Phi$ is unimodular.

Consider $$F(q,x)=q^{N}
\left(\prod_{k=1}^N\frac{e^{i\langle \alpha_k,x\rangle}-1}{e^{iq\langle \alpha_k,x\rangle}-1}\right).$$
Denote by $F([q],x)$ the Taylor series at $q=0$ of $F(q,x)$ and write
$$F([q],x)=\int_V e^{i\ll x,v\rr}m([q])(v)dv$$
where $m([q])=\sum_{a=0}^{\infty} q^a m_a$ is a series of generalized functions on $V$.

$\bullet$ Then $m([q])\in \mathcal S[[q]]$.

$\bullet$  All, but a finite number, the generalized functions $m_a$  are supported on walls.
 Thus, for any generic vector $\epsilon$, the  series $\lim_{\epsilon}^{\Lambda} m([q])$ is a polynomial in $q$.

$\bullet$
If $\epsilon$ is a generic vector {\bf belonging to the cone} $Cone(\Phi)$ generated by the elements of $\Phi$,
 $$\lim_{\epsilon}^{\Lambda}m([q])|_{q=1}=\delta^{\Lambda}_0.$$
\end{theorem}

\subsection{Inversion of the Box spline with parameters}\label{subinv}

To explain what happens when  $\Phi$ is not unimodular, we need more notations. We also introduce parameters.
We assume that $\Phi$ spans $V$.

Let $\Gamma\subset U$ be the dual lattice of the lattice $\Lambda$. Thus $\ll \gamma,\lambda\rr\in \Z$ if $\gamma\in \Gamma$ and $\lambda\in \Lambda$.
We consider $\Lambda$ as the group of characters of the torus $T=U/2\pi \Gamma$,
and use the notation $s^\lambda$ for the value
of $\lambda \in \Lambda$ at $s\in T$.
If $S\in U$ is a representative of $s\in U/2\pi \Gamma$, then, by definition, $s^{\lambda}=e^{i\ll S,\lambda\rr}$.

We denote by $\hat s$ the character $\lambda\mapsto s^{\lambda}$ of $\Lambda$.
If $m$ is a function on $\Lambda$, then $\hat s m$ is a function on $\Lambda$: $\hat s m(\lambda)=s^{\lambda}m(\lambda)$.

For $s\in T$, let
 $$\Phi(s)=\{\alpha\in \Phi; s^{\alpha}=1\}.$$

\begin{definition}\label{def:cv}
Let $\CV$ be the subset of $T$ consisting of the elements $s$ such that $\Phi(s)$ still spans $V$.
\end{definition}

Thus $\CV$ is a finite subset of $T$ called the vertex set.

Let ${\bf y}=[y_1,y_2,\ldots,y_N]$ be a list of $N$ complex numbers.
The box spline $B(\Phi,\y)$, with parameter $\y$,  is the measure on $V$ such that, for a  continuous function  $F$ on $V$,
$$
\langle B(\Phi,\y),F\rangle=
\int_{0}^{1}\cdots \int_{0}^{1} e^{i(\sum_{k=1}^N t_k y_k)}
F(\sum_{k=1}^{N} t_k\alpha_k) dt_1\cdots dt_N.
$$

The Fourier transform  ${\hat B}(\Phi,\y)(x)$  is the function
 $${\hat B}(\Phi,\y)(x)=\prod_{k=1}^N\frac{e^{i(\langle \alpha_k,x\rangle+y_k)}-1}{i(\langle \alpha_k,x\rangle+y_k)}.$$

If $s\in T$,  we define
$$F_s(q,x,\y)=q^{|\Phi(s)|}
\left(\prod_{k=1}^N\frac{e^{i(\langle \alpha_k,x\rangle+i y_k)}s^{\alpha_k}-1}{e^{iq(\langle \alpha_k,x\rangle+ y_k)}s^{\alpha_k}-1}\right).$$

We denote by $F_s([q],x,\y)$ the Taylor series of
$F_s(q,x,\y)$ at $q=0$, a series of analytic functions of $x$, depending of the parameter $\y$.

Write
$$F_s([q], x,\y)=\int_V e^{i\ll x,v\rr} m_s([q],\y)(v)dv$$
where $m_s([q],\y)$ is a series of generalized functions on $V$.

\begin{theorem}\label{deconvtrans}
$\bullet$ The series $m_s([q],\y)$ of generalized functions on $V$ is in $\CS^{\y}[[q]]$.

$\bullet$
If $s$ is not in $\CV$, the
 generalized function $m_s([q],\y)$ is supported on walls.

If $\y$ is sufficiently small, and $\epsilon$ is a generic vector, the series
$\lim_{\epsilon}^{\Lambda} m_s([q],\y)$ is convergent at $q=1$.

 $\bullet$ Assume $\y$ is sufficiently small.
  If $\epsilon$ is a generic vector {\bf belonging to the cone} $Cone(\Phi)$ generated by the elements of $\Phi$,  then

  \begin{equation}\label{eqdeconv}
\sum_{s\in \CV} {\hat s}^{-1} \lim_{\epsilon}^{\Lambda} m_s([q],\y)|_{q=1}=\delta^{\Lambda}_0.
\end{equation}

\end{theorem}

It is easy to see (see Formula (\ref{eq;Fandboxy}))   that
 $m_s([q],\y)$ is obtained by an explicit expression in terms of derivatives and translates of the Box spline $B(\Phi(s),\y)$.
 Thus, in particular, it is supported on affine walls if $s$ is not in $\CV$.

Equation (\ref{eqdeconv}) is clearly equivalent to the identity:
 \begin{equation}\label{eqdeconvfourier}
\sum_{\lambda\in \Lambda}(\sum_{s\in \CV} {\hat s}^{-1} \lim_{\epsilon} m_s([q],\y)(\lambda)) e^{i\ll \lambda,x\rr}|_{q=1}=1.
\end{equation}

\begin{example}\label{third}
We verify Theorem \ref{deconvtrans} in the  cases of Examples \ref{first}.

We use the following Taylor series expansions at $z=0$.

\begin{equation}\label{expansion1}
\frac{z}{e^z-1}=\sum_{a=0}^{\infty} b(a) z^a/a!
\end{equation}
where $b(a)$ are the Bernoulli numbers.

%We  will also need the Taylor expansion
%\begin{equation}\label{expansion2}
%\frac{1}{e^zu-1}=\sum_{a=0}^{\infty} \beta(a,u) z^a/a!.
%\end{equation}
%
 %Here $u\neq 1$, and $\beta(a,u)$ are constants depending on $u$.
%
%

$\bullet$

 $\Phi_1=[\omega]$, $\y=[y]$.
 Then $\CV=\{1\}$.

Then
$$F_1(q,x,\y)=q\frac{e^{i(x+y)}-1}
{e^{iq(x+y)}-1}=\frac{e^{i(x+y)}-1}
{i(x+y)} \frac{iq(x+y)}
{e^{iq(x+y)}-1}.$$

The Taylor series in $q$ is

$$F_1([q],x,\y)=
\frac{e^{i(x+y)}-1}{i(x+y)}\sum_{a=0}^{\infty} q^a b(a) \frac{(i(x+y))^a}{a!}.$$

We write $F_1([q],x,\y)=\frac{e^{i(x+y)}-1}{i(x+y)}+\sum_{a=1}^{\infty}f_a(x,\y)q^a$
where $$f_a(x,\y)=(e^{i(x+y)}-1) b(a) \frac{(i(x+y))^{a-1}}{a!}$$
and take the Fourier transform:
 $F_1([q],x,\y)=\int_\R e^{it x} m_1([q],\y)(t)dt$
with $m_1([q],\y)=\sum_{a=0}^{\infty} q^a m_a(\y).$

As, for $a\geq 1$,  $\frac{(i(x+y))^{a-1}}{a!}$ is a polynomial function of $x$,   the Fourier transform $m_a(\y)$ of  $f_a(x,\y)$ is supported on $\Z$ (more precisely on $\{0,1\}$).

Then
$$m_1([q],\y)=B(\Phi_1,\y)+\sum_{a=1}^{\infty} q^a m_a(\y)$$
and $m_a(\y)$ for $a\geq 1$  restricts  to  $0$  on $\R\setminus \Z$. We obtain that the restriction of $m_1([q],\y)$ to
 $\R\setminus \Z$ is given by

$$m_1([q],\y)(t)=\begin{cases}0\ \  &\text{if\ } t<0 \\ e^{ity} \ \  &\text{if\ } 0<t<1\\
0\ \  &\text{if\ } t>1
.\end{cases}$$
Thus we see that the right limit at all elements of $\Lambda$ is $0$, except at $\lambda=0$, where it is $1$.
 This prove that the right limit $\lim_{\epsilon}^{\Lambda}m_1([q],\y)|_{q=1}$ is equal to  $\delta^{\Lambda}_0$.
Theorem \ref{deconvtrans} would not be true with left limits.

$\bullet$  $\Phi'_2=[\omega, -\omega]$, $\y=[y_1,y_2]$.
Again, $\CV=\{1\}$.
We have
$$F_1(q,x,\y)=q^2\left(\frac{e^{i(x+y_1)}-1}{e^{iq(x+y_1)}-1}\right)\left(\frac{e^{i(-x+y_2)}-1}{e^{iq(-x+y_2)}-1}\right).$$

Thus $F_1([q],x,\y)$ is equal to
$$\left(\frac{e^{i(x+y_1)}-1}{i(x+y_1)}\right)
\left(\frac{e^{i(-x+y_2)}-1}{i(-x+y_2)}\right)
\left(\sum_{a=0}^{\infty} q^a b(a) (i(x+y_1))^a/a!\right)
\left(\sum_{\ell=0}^{\infty} q^\ell b(\ell) (i(-x+y_2))^\ell/\ell!\right).$$

We write $F_1([q],x,\y)=\int_{\R}e^{itx}m_1([q],\y)(t)dt$.
By Fourier transform, $i(x+y_1)$ acts by $iy_1-\partial_t$ and this operator annihilates the function $e^{ity_1}$ while
$$\left(\sum_{a=0}^{\infty} q^a b(a) (iy_1-\partial_t)^{a}/a!\right)\cdot e^{-it y_2}=\frac{q(i(y_1+y_2))}{e^{i[q](y_1+y_2)}-1} e^{-it y_2},$$
the series being convergent at $q=1$ for $\y$ small.
Thus using Formulae for $B(\Phi'_2,\y)$,
we obtain that the restriction to $\R\setminus \Z$ of the Fourier transform $m_1([q],\y)$  of $F_1([q],x,\y)$ is given by

 $$m_1([q],\y)(t)=\begin{cases} 0\ \  &\text{if\ } t<-1 \\
\\
qe^{i(y_1+y_2)}\frac{e^{it y_1}}{e^{i[q](  y_2+ y_1)}-1}-
q\frac{e^{-it y_2}}{e^{i[q](y_1+y_2)}-1}\ \ \  &\text{if\ } -1<t<0\\
 \\ q e^{(i(y_2+y_1))}
\frac { e^{-i ty_2}}{e^{i[q](y_2+y_1)}-1}-q
\frac { e^{i ty_1}}{e^{i[q](y_1+y_2)}-1}
\ \  &\text{if\ } 0<t<1\\
\\
0\ \  &\text{if\ } t>1 \\
\end{cases}$$

We verify that
$m_1([1],\y)(t)$ is continuous and that its restriction to $\Lambda$ is equal to $\delta^{\Lambda}_0$.
 As asserted by Theorem \ref{deconvtrans},
  we can take limits from the left or right, as the cone generated by $\Phi'_2$ is equal to $\R$.

$\bullet$  $\Phi=[\omega, 2\omega]$, $\y=[y_1,y_2]$.

Then $\CV=\{s=1,s=-1\}$.

We have

$$F_1(q,x,\y)=q^2\left(\frac{e^{i(x+y_1)}-1}{e^{iq(x+y_1)}-1}\right)\left(\frac{e^{i(2x+y_2)}-1}{e^{iq(2x+y_2)}-1}\right).$$

We write
 $F_1([q],x,\y)=\int_{\R}e^{itx}m_1([q],\y)(t)dt$.

The restriction to  $\R\setminus \Z$ of $m_1([q],\y)(t)$ is given by

$$m_1([q],\y)(t)=\begin{cases} 0\ \  &\text{if\ } t<0 \\
\\
q(\frac{e^{it y_1}}{(1-e^{i [q](y_2- 2 y_1) })}-\frac{1}{2}\frac{e^{it\frac{ y_2}{2}}}{(e^{i [q]( y_1- \frac{y_2}{2}) }-1)})\ \ \  &\text{if\ } 0<t<1\\
 \\
\frac{1}{2} q\frac {e^{it\frac{ y_2}{2}} (e^{i (y_1-\frac{y_2}{2})}-1)}{(e^{i [q]( y_1- \frac{y_2}{2}) }-1) }\ \ &\text{if\ } 1<t<2\\
\\
-q\frac{e^{it y_1}e^{i (y_2-2y_1)}}{(1-e^{i [q]\left(y_2- 2 y_1 \right) })}+\frac{1}{2} q\frac{e^{it\frac{ y_2}{2}}e^{i (y_1-\frac{y_2}{2})}}
{(e^{i [q]\left(  y_1- \frac{y_2}{2} \right) }-1)}
\ \  &\text{if\ } 2<t<3\\\\
0\ \  &\text{if\ } t>3 \\
\end{cases}$$

We  consider now the case where $s=-1$. Then
$$F_{-1}(q,x,\y)=q\left(\frac{e^{i(x+y_1)}+1}{e^{iq(x+y_1)}+1}\right)\left(\frac{e^{i(2x+y_2)}-1}{e^{iq(2x+y_2)}-1}\right).$$

We write
 $F_{-1}([q],x,\y)=\int_{\R}e^{itx}m_{-1}([q],\y)(t)dt$.

The restriction to $\R\setminus\Z$ of $m_{-1}([q],\y)$
is given by

$$m_{-1}([q],\y)(t)=\begin{cases} 0\ \  &\text{if\ } t<0 \\
\\
\frac{1}{2}\frac{e^{it \frac{y_2}{2}}}{(1+e^{i [q](y_1-y_2/2)})} \ \ \ &\text{if\ } 0<t<1\\
 \\
\frac{1}{2} e^{it\frac{ y_2}{2}}
\frac{1+e^{i (y_1-y_2/2)}}{1+e^{i [q](y_1-y_2/2)}}  \ \  &\text{if\ } 1<t<2\\
\\
\frac{1}{2}\frac{e^{it y_2/2}e^{i (y_1-y_2/2)}}{(1+e^{i [q]\left(y_1- y_2/2 \right) })}
\ \  &\text{if\ } 2<t<3\\\\
0\ \  &\text{if\ } t>3 \\
\end{cases}$$

We verify that, for $\epsilon>0$

$$\lim_{\epsilon}(m_1([1],\y))(n)+(-1)^n\lim_{\epsilon}(m_{-1}([1],\y))(n)=
\begin{cases}1\ \  &\text{if\ }n=0 \\ \\ 0 \ \  &\text{if\ } n\neq 0\\ \end{cases}$$

\end{example}

This formula is not true for $\epsilon<0$.

\subsection{Translated Box spline}\label{subtranslated}

It is quite natural to introduce translated Box splines.

Assume that $\Phi$ spans $V$.
Let $\y=[y_1,y_2,\ldots,y_N]$ be a sequence of complex numbers, as before,
and let  $r$ be a point of $V$. We choose $\r=[r_1,r_2,\ldots,r_N]$  a sequence of real numbers
 so that $r=\sum_{k=1}^N r_k \alpha_k$.
   We say that $\r$ is a $\Phi$-representation of $r$.
   Let $\ll\r,\y\rr=\sum_{k=1}^N r_k y_k$.

\begin{definition}
We define
$$B_r(\Phi,\y)(v)=e^{-i\ll\r,\y\rr}B(\Phi,\y)(v+r).$$
\end{definition}

Although $B_r(\Phi,\y)$ depends of the representation of $r$ as
 $r=\sum_{k=1}^N r_k \alpha_k$, we do not include this in the notation.

The support of $B_r(\Phi,\y)$ is $Z(\Phi)-r$. Thus $0$ is in the support of $B_r(\Phi,\y)$ if and only if $r$ belongs to the zonotope $Z(\Phi)$.
If $r\in Z(\Phi)$, we denote by $Cone(r,\Phi)$ the tangent cone at $r$ to the zonotope $Z(\Phi)$.

The Fourier transform ${\hat B}_r(\Phi,\y)(x)=e^{-i(\ll r,x\rr+\ll \r,\y \rr)}
{\hat B}(\Phi,\y)(x)$ is thus equal to

$$e^{-i(\ll r,x\rr+\ll \r,\y \rr)}\prod_{k=1}^N \frac{e^{i(\ll \alpha_k,x\rr+y_k)}-1}{i(\ll \alpha_k,x\rr+y_k)}
=\prod_{k=1}^N \frac{e^{i (1-r_k) (\ll \alpha_k,x\rr+y_k)}-e^{-ir_k(\ll \alpha_k,x\rr+y_k)}}{i(\ll \alpha_k,x\rr+y_k)}.$$

A natural point of translation is the center of the Box spline $\rho=\frac{1}{2}\sum_{k=1}^N \alpha_k$ represented by $\r_\rho=[\frac{1}{2},\ldots,\frac{1}{2}]$. Then
  $ B_{\rho}(\Phi,\y)$ has Fourier transform
$$\prod_{k=1}^N \frac{e^{i \frac{ (\ll \alpha_k,x\rr+y_k)}{2}} -e^{-i \frac{ (\ll \alpha_k,x\rr+y_k)}{2}}}{i(\ll \alpha_k,x\rr+y_k)}.$$

However, we can consider any point $r\in V$.

 We write $B_r(\Phi,\y)=b_r(\Phi,\y)(v)dv$.
We define $V_{reg,r}=V_{reg}-r$ which is the disjoint union of translated alcoves $\c-r$, with boundaries the translated walls.
We construct similarly the space $PW_r$ of piecewise polynomial functions on $V_{reg,r}$,
 the space  $PW^{\omega}_r$ of piecewise analytic functions on $V_{reg,r}$.
 Then $b_r(\Phi,\y)\in PW^{\omega}_r$.
 We denote by $\CS_r^{\y}$ the space of derivatives by constant coefficients differential operators of $b_r(\Phi,\y)$. Spaces $PW_r,PW^{\omega}_r,\CS^{\y}_r$  are isomorphic to the space $PW, PW^{\omega},\CS^{\y}$ by the translation
$$(\tau(r)f)(v)=f(v+r).$$
If $\epsilon$ is a generic vector, we can still take the limit on $\Lambda$ (Note: we do not translate our lattice) of a function in $\CS_r^{\y}$:

$$\lim_\epsilon^{\Lambda}:\CS_r^{\bf y}\to {\mathcal C}(\Lambda).$$

If $s\in T$,  we define
$$F_s(q,\r,x,\y)=q^{|\Phi(s)|}e^{i(q-1)(\ll r,x\rr+\ll \r,\y \rr)}\left(\prod_{k=1}^N\frac{e^{i(\langle \alpha_k,x\rangle+i y_k)}s^{\alpha_k}-1}{e^{iq(\langle \alpha_k,x\rangle+ y_k)}s^{\alpha_k}-1}\right)$$
$$=q^{|\Phi(s)|}\left(\prod_{k=1}^N\frac{e^{-ir_k(\langle \alpha_k,x\rangle+y_k)}(e^{i(\langle \alpha_k,x\rangle+i y_k)}s^{\alpha_k}-1)}{e^{-iqr_k(\langle \alpha_k,x\rangle+y_k)}(e^{iq(\langle \alpha_k,x\rangle+ y_k)}s^{\alpha_k}-1)}\right).$$

When $q=1$, $F_s(q,r,x,\y)=1$.

We denote by $F_s([q],r,x,\y)$ the Taylor series of
$F_s(q,r,x,\y)$ at $q=0$.

Write
$$F_s([q],r, x,\y)=\int_V e^{i\ll x,v\rr} m_s([q],r,\y)(v)dv$$
where $m_s([q],r,\y)$ is a series of generalized functions on $V$.
The following theorem is the main theorem of this article.

\begin{theorem}\label{theo:deconvtrans}
$\bullet$ The series $m_s([q],\r,\y)$ of generalized functions on $V$ is in $\CS_r^{\y}[[q]]$.

$\bullet$
If $s$ is not in $\CV$, the
 generalized function $m_s([q],\r,\y)$ is supported on translated walls.

If $\y$ is sufficiently small, and $\epsilon$ is a generic vector, the series
$\lim_{\epsilon}^{\Lambda} m_s([q],\r,\y)$ is convergent at $q=1$.

 $\bullet$ Assume $\y$ is sufficiently small.
 If $r$ {\bf belongs to the zonotope}, and if
   $\epsilon$ is a generic vector {\bf belonging to the cone} $Cone(r,\Phi)$ tangent at $r$ to the zonotope  $Z(\Phi)$,

  \begin{equation}\label{eqdeconvr}
\sum_{s\in \CV} {\hat s}^{-1} \lim_{\epsilon}^{\Lambda} m_s([q],\r,\y)|_{q=1}=\delta^{\Lambda}_0.
\end{equation}

\end{theorem}

\begin{remark}\label{remind}

$\bullet$
The function
$F_s([q],\r,x,\y)$ is equal to $$e^{i ([q]-1)\ll r,x\rr} e^{i([q]-1)(\ll \r,\y\rr)}F_s([q],x,\y).$$
When we evaluate at $q=1$, we see that  $\ll\r,\y\rr$ plays no role.
In particular, Theorem \ref{theo:deconvtrans} does not depend of the way $r$ is represented as $\sum_{k=1}^N r_k\alpha_k$.

$\bullet$
The series of functions $m_s([q],\r,\y)$   is obtained
   from the series of functions $m_s([q],\y)$ (defined in the preceding subsection, Subsection \ref{subinv})
  by a rather amusing operation.
On each alcove $\c$, the series of functions $m_s([q],\y)$ is given by  the restriction to $\c$ of a series of  analytic functions
  $m_s^{\c}([q],\y)$ defined on all $V$.
  Then we see that  $m_s([q],\r,\y)$ on $\c-r$ is just equal to the series
   $e^{i([q]-1)\ll \r,\y\rr} m_s^{\c}([q],\y)$, restricted to $\c-r$. Indeed the
   Fourier transform  of the operator $e^{i ([q]-1)\ll r,x\rr}$ acts by on $f\in PW_r^{\omega}$ by taking the Taylor series  of $f$ at $v+r$,  evaluated at $v+r-qr$.
   It is NOT the identity at $q=1$, as $f$ is not analytic on $V$!!.

\end{remark}

Let us consider the corresponding functions for $\Phi'_2$ and $\y=0,q=1$ (Figure \ref{figuremove}).
We  see that these functions restricted to the lattice are still equal to $\delta_0^{\Lambda}$, under the condition that $-1<r<1$.
But when $|r|>1$, then $0$  is not any more on the support, so the restriction cannot  be $\delta_0^{\Lambda}$.

\begin{figure}\label{figuremove}
\begin{center}
   \includegraphics[width=3cm]{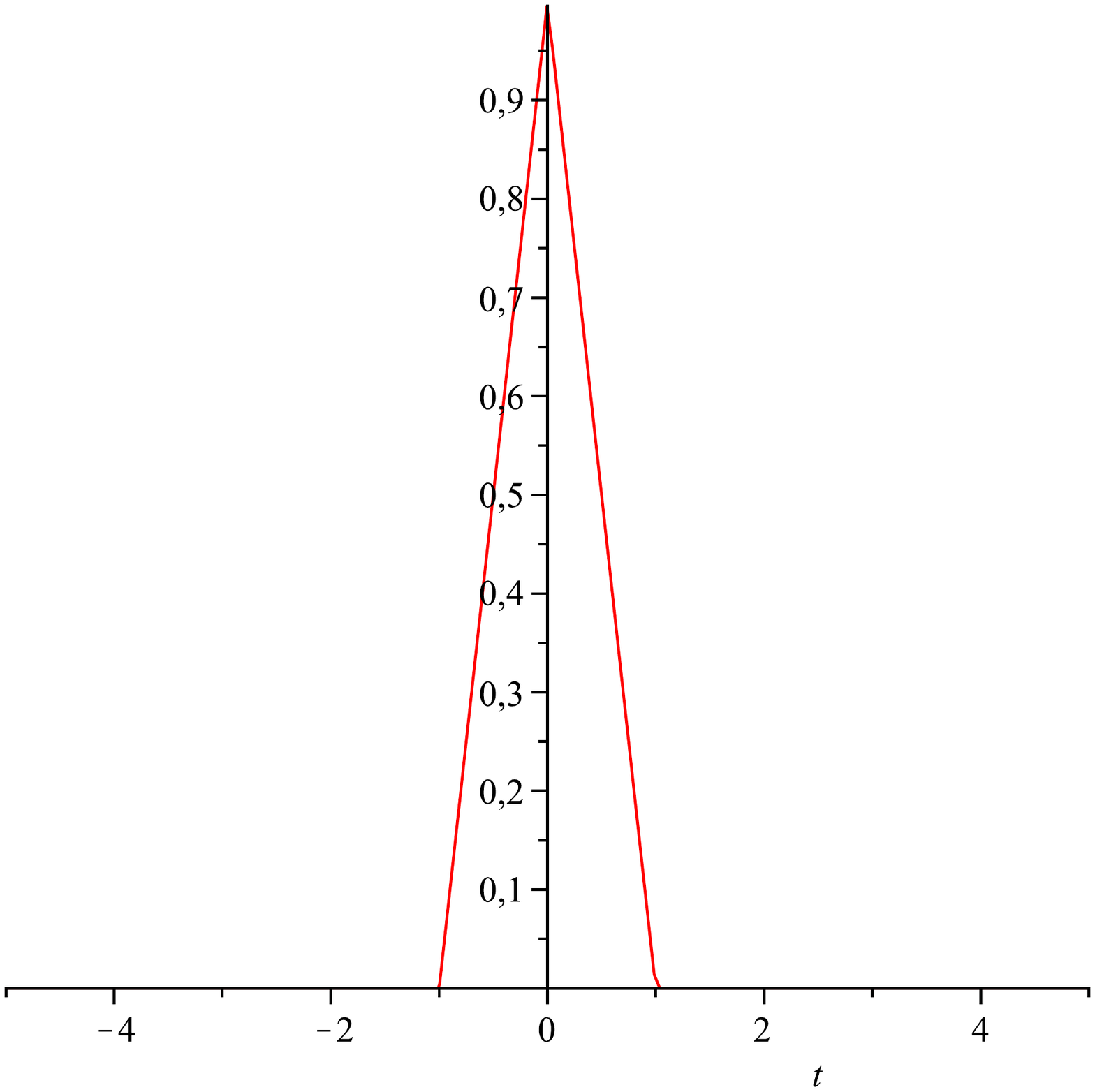}   \includegraphics[width=3cm]{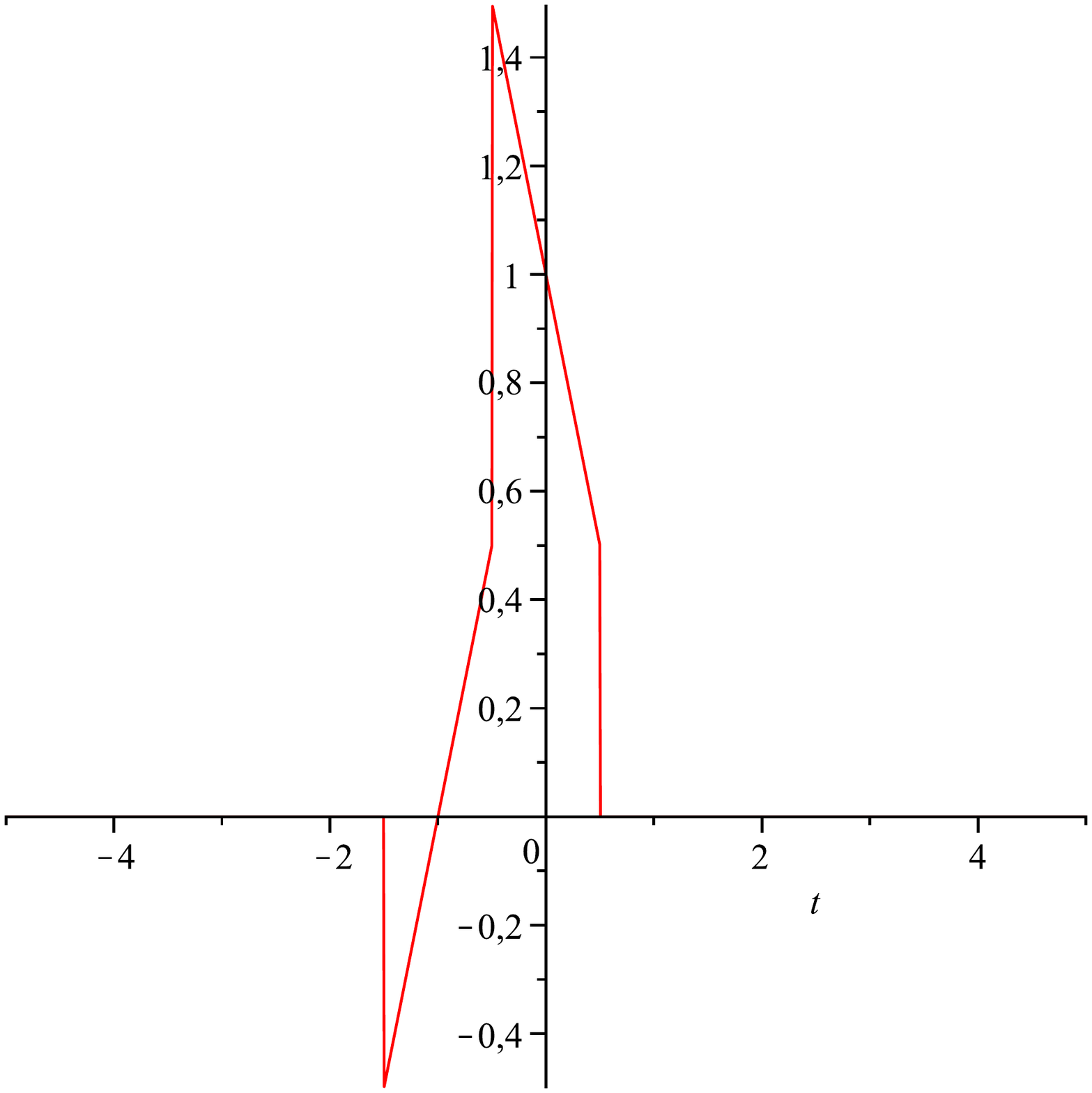}\includegraphics[width=3cm]{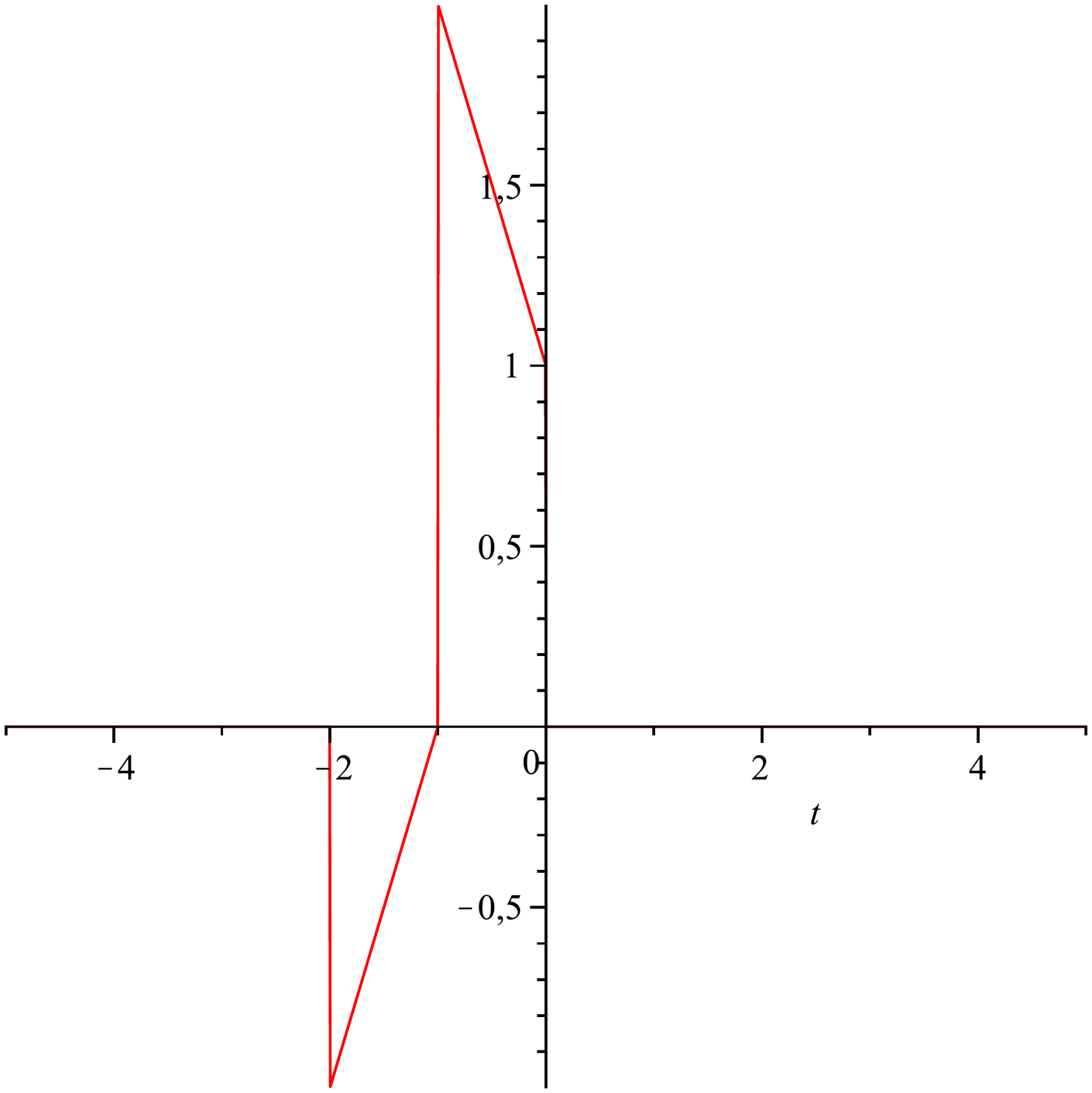}
\includegraphics[width=3cm]{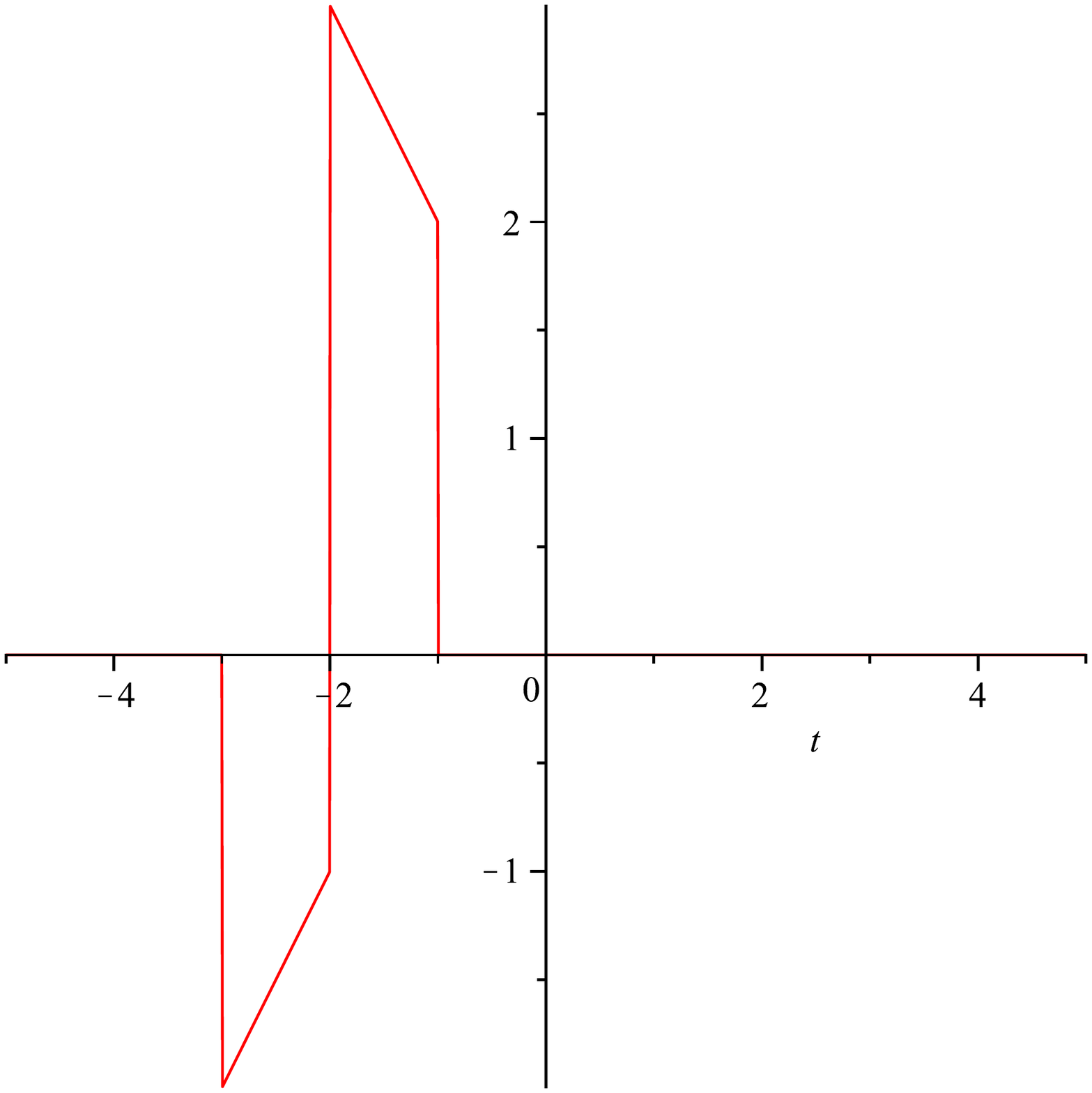}  \caption{$r=0$,\, $r=\frac{1}{2}$, \, $r=1$,\, $r=2$}
  \end{center}
\end{figure}

For $\y=0$, Theorem  \ref{theo:deconvtrans} is equivalent to  our description of the local pieces
$m_s^{\c}$ of the functions $m_s$, in \cite{dpv1}.
We will show this point in  Theorem \ref{theosameasdpv}.

\section{Rational functions on $U$ and  functions on $T$ }\label{secillumin}

We consider a system $\Phi$ spanning $V$, and contained in a lattice $\Lambda$ with dual lattice $\Gamma$.
We follow notations of Subsection \ref {subtranslated}.
We denote by  $\CH$   the space of holomorphic functions  on $U\times \C^N$.

Recall the definition of $\CV$ (Definition \ref{def:cv}).
\begin{definition}\label{def:hatcv}
We denote by $\hat \CV$ the reciproc image
of $\CV$  in $U$.
\end{definition}
The subset  $\hat \CV$  of $U$  contains $2\pi \Gamma$ and is a finite union of cosets of $2\pi \Gamma$.
Remark also that $\hat \CV$ depends only of the list $\Phi$ and not of the lattice $\Lambda$ containing the list $\Phi$.
If $\Phi$ is unimodular, then $\hat \CV=2\pi \Gamma$.

Denote  by $\overline{u}$ the image of $u\in U$ in $U/2\pi \Gamma$.  Thus $\Phi({\overline u})=\{\alpha_k, e^{i\ll \alpha_k,u\rr}=1\}$.

\begin{definition}
Define
$$\Theta(\r, x,\y)=
\frac{e^{i(\ll r,x\rr+\ll \r,\y\rr)}}{\prod_{k=1}^N(e^{i(\langle \alpha_k,x \rangle+y_k)}-1)}=
\prod_{k=1}^N
\frac{e^{ir_k(\langle \alpha_k,x \rangle+y_k)}}{e^{i(\langle \alpha_k,x \rangle+y_k)}-1}.$$
\end{definition}

The function $\Theta(\r, x,\y)$ satisfies the following covariance properties.

\label{coinv}
$\bullet$
If  $\gamma\in \Gamma$, then

 $$\Theta(\r, x+2\pi \gamma,\y)=e^{i\ll r,2\pi \gamma\rr}  \Theta(\r, x,\y).$$

$\bullet$  If $\r=\r'+[a_1,a_2,\ldots,a_N]$,  then
$$\Theta(\r, x,\y)=e^{i(\ll a,x\rr+\ll {\bf a},\y\rr)}\Theta(\r',x,\y)$$
with $a=\sum_{k=1}^N a_k \alpha_k$.

\bigskip

For $w\in V$, define
$$Z(q,\r,w)(x,\y)=
q^{|\Phi(\overline{w})|}
e^{i (q-1)(\ll r,x-w\rr+\ll \r,\y\rr)}\prod_{k=1}^N
\frac{e^{i(\langle \alpha_k,x\rangle +y_k)}-1}
{e^{iq(\langle \alpha_k,x-w\rangle +y_k)}e^{i\langle \alpha_k,w\rangle }-1}.
$$
%
%$$=q^{|\Phi(\overline{w})|}
%\prod_{k=1}^N
%\frac{e^{-ir_k(\langle \alpha_k,x-w \rangle+y_k)}(e^{i(\langle \alpha_k,x\rangle +y_k)}-1)}
%{e^{-iqr_k(\langle \alpha_k,x-w \rangle+y_k)}(e^{iq(\langle \alpha_k,x-w\rangle +y_k)}e^{i\langle \alpha_k,w\rangle }-1)}
%.$$

If $q=1$, $Z(q,\r,w)=1$.

Let
$Z([q],\r,w)$ be the Taylor series of $Z(q,\r,w)$.
This is a series of holomorphic functions of $w,x,\y$, and we write
$Z([q],\r,w)(x,\y)=\sum_{a=0}^{\infty}q^a z_a(w,x,\y)$.
When $(x,\y)$ varies in a compact subset of $U_\C\times \C^N$, the functions
$z_a(w,x,\y)$ are of at most polynomial growth on $w$ , on each coset of  $2\pi \Gamma$  (this will be proven in Lemma \ref{lemtechnical}).
We can then define the following sum in the sense of generalized functions.

\begin{definition}\label{defCB}
$$\CB([q],\r)(v)(x,\y)=\sum_{w\in {\hat \CV}}Z([q],\r,w)(x,\y)e^{ i\langle v,w\rangle }.$$
\end{definition}

This sum has a meaning as a generalized function of $v$ with coefficients in $\CH$.

Then, we have
\begin{theorem}\label{delicate}

$\bullet$ For $v\in V_{reg,r}$, the function
$(v,x,\y)\mapsto \CB([q],\r)(v)(x,\y)$ is a formal series of  analytic functions of $(v,x,\y)$.

$\bullet$ Let $\y$ be small enough.
 If $\epsilon$ is a generic vector,   then
$\lim_{t>0,t\to 0} \CB([q],\r)(t\epsilon)(x,\y)$ is a series of analytic functions in $(x,\y)$  convergent for $q=1$.

Furthermore, if $r$ {\bf is in the zonotope} $Z(\Phi)$ and if $\epsilon$
{\bf belongs to the cone} $Cone(r,\Phi)$ tangent at $r$ to $Z(\Phi)$, then
$$\lim_{t>0,t\to 0} \CB([q],\r)(t\epsilon)(x,\y)|_{q=1}=1.$$
\end{theorem}
As we will see, this theorem is equivalent to Theorem \ref{theo:deconvtrans}.

\bigskip

We reformulate this theorem by using meromorphic functions.
Let
 $$Q(q,\r,w,x,\y)= \left(q^{|\Phi(\overline{w})|} \frac{e^{iq(\ll r,x-w\rr+\ll \r,\y\rr)}e^{i\ll r,w\rr}}
{\prod_{k=1}^N(e^{iq(\langle \alpha_k,x-w\rangle +y_k)}e^{i\langle \alpha_k,w\rangle}-1)}\right),
$$
that is
\begin{equation}\label{eq;CQ}
Q(q,\r,w,x,\y)=\Theta(r, qx+(1-q)w,q\y).
\end{equation}

We have
$$Q(q,\r,w,x,\y)=Z(q,\r,w,x,\y)\Theta(\r,x,\y).$$

We  consider the series
 $$\CT([q],\r)(v)(x,\y)=\sum_{w\in \hat\CV}
Q([q],\r,w,x,\y)e^{i\ll v,w\rr}.$$

Then equivalently,

\begin{theorem}\label{theo;thetalim}
Let $\y$ be small enough.
 If $\epsilon$ is a generic vector,   then
$$\lim_{t>0,t\to 0} \CT([q],\r)(t\epsilon)(x,\y)$$ is a series of meromorphic functions, convergent for $q=1$.
Furthermore, if $r$ {\bf  is in the zonotope} $Z(\Phi)$ and if $\epsilon$
{\bf belongs to the cone} $Cone(r,\Phi)$ tangent at $r$ to $Z(\Phi)$, then
$$\lim_{t>0,t\to 0} \CT([q],\r)(t\epsilon)(x,\y)|_{q=1}=\Theta(\r,x,\y).$$
\end{theorem}

Let us explain the philosophy of this theorem and, for this purpose, it is sufficient to consider the case $\y=\r=0$, and the unimodular case where $\hat\CV=2\pi \Gamma$.
Consider the function
$$\Theta(x)=\prod_{k=1}^N
\frac{ 1}{e^{i\langle \alpha_k,x \rangle}-1}.$$
It is a function of $x$ invariant by translation by an element of $2\pi \Gamma$: $\Theta(x+2\pi \gamma)=\Theta(x)$.
Consider the Laurent series $T(x)$ of $\Theta(x)$ at $x=0$, up to  order $N-1$. This is a  rational function of $x$ decreasing at $\infty$.
If we sum $T(x)$  over the coset  $x-2\pi \Gamma$, we reobtain a periodic function of $x$, with same Taylor series at $x=0$ than $\Theta(x)$ up to order $N-1$.
However,  the summation is not absolutely convergent.
Thus we  introduce  the oscillatory term $e^{2i\pi \ll v,\gamma\rr}$ and consider instead
$$\CT([q])(v)(x)=\sum_{\gamma\in  \Gamma} q^{N}\prod_{k=1}^N \left(\frac{1}{e^{i[q]\langle \alpha_k,x-2\pi\gamma \rangle}-1}\right)e^{2i \pi\ll v,\gamma\rr}$$
a formal sum of rational functions of $x$, converging in the distribution sense (in $v$).
Theorem \ref{th:Thetadelicateuni} asserts that when $q=1$ and $v$ tends to $0$ in appropriate directions, we  recover $\Theta(x)$.

It is enlightening, and needed  for our proof by induction, to give the full proof of Theorem \ref{theo:deconvtrans}  in the simplest case $\Phi=[\omega]$ in $V=\R \omega$.
 The following well known lemma is the heart of the proof.

Let $v\in \R$ and  let  $[v]$ denotes the integral part of $v$. Then the function $\{v\}=v-[v]$ is a periodic function of $v$, and $\{v\}=v$ when $0<v<1$.
\begin{lemma}\label{bernoulli}
We have the equality of $L^2$-functions of  $v\in \R/\Z$:
$$\sum_{n\in \Z}  \frac{e^{ix}-1}{i(x-2\pi n)} e^{2i\pi n v}= e^{i\{v\}x}.$$
\end{lemma}
\begin{proof}
Indeed, let us compute the $L^2$-expansion of the periodic function
$v\mapsto e^{i\{v\}x}$ on $\R/\Z$.
By definition,  this  is
$$\sum_{n\in \Z} \left(\int_{0}^1e^{i\{v\}x} e^{-2i\pi n v} dv\right) e^{2i\pi n v}=
\sum_{n\in \Z} \left(\int_{0}^1 e^{iv(x-2\pi n)} dv\right) e^{2i\pi nv}
$$
$$= \sum_{n\in \Z}  \frac{e^{i(x-2\pi n)}-1}{i(x-2\pi n)} e^{2i\pi nv}=\sum_{n\in \Z}  \frac{e^{ix}-1}{i(x-2\pi n)} e^{2i\pi nv}.$$
\end{proof}

Thus consider $V=\R\omega$, $\Phi=[\omega]$, $\y=[y]$ and $\r=[r]$ and let
 $$\Theta(\r,x,\y)=\frac{e^{i r(x+y)}}{e^{i(x+y)}-1},$$
 $$Q(q,r,w,x,\y)=e^{i qr(x-w+y)}e^{irw}\frac{q}{e^{i q(x-w+y)}e^{iw}-1}.$$

We have $\hat\CV=2\pi \Z$.
We see that the Taylor series
$Q([q],r,2\pi n,x,\y)=\sum_{a=0}^{\infty} q^a z_a(x,\y,n)$ of $Q(q,r,2\pi n,x,\y)$
is of the form
$$e^{2i\pi nr} \left(\frac{1}{i (x-2\pi n+y)}+
\sum_{a=1}^{\infty} q^a p_a(n,x,y)\right)$$
where $p_a(n,x,y)$  is polynomial in $n,x,y$.
So the sum $\sum_{n\in \Z}e^{2i\pi nr}p_a(n,x,y)e^{2i\pi nv}$  is
a distribution of $v$ supported on $v+r\in \Z$.
Thus, by Lemma   \ref{bernoulli},  the series $\sum_{n\in \Z} Q([q],r,2\pi n,x,\y)e^{2i\pi nv}$ restricted to $v+r\notin \Z$ is equal to $$\frac{e^{i\{v+r\}(x+y)}}{e^{i(x+y)}-1}.$$

If $r$  is not an integer, we obtain  when $v$ tends to $0$ that the limit is $$\frac{e^{i(\{r\})(x+y)}}{e^{i(x+y)}-1}.$$

If $0<r<1$ , this is $\Theta(\r,x,\y)$.

If $r=0$,  the right limit of
$\{v+r\}$ when $v$ tends to $0$ is $r$.

If $r=1$,
the left limit of
$\{v+r\}$ when $v$ tends to $0$ is $r$.

Thus Theorem \ref{theo;thetalim} holds, if and only if $r, \epsilon$ verifies the conditions stated in the theorem.

\part{Proofs}
Our strategy to prove Theorem \ref{theo:deconvtrans} is to apply  Poisson formula, and  see that Theorem \ref{theo:deconvtrans} is equivalent to Theorem
\ref{delicate}.
Then we  prove Theorem
\ref{delicate} by induction on the number of elements of $\Phi$.

\section{Poisson formula for  derivatives of splines}
Let $f$ be a smooth function on $V$ with compact support.
 By Poisson formula, we have the equality for $(x,v)\in U\times V$:

 \begin{equation}\label{classicpoisson}
 \sum_{\lambda\in \Lambda} f(\lambda+v)e^{i \ll x,\lambda\rr}=\sum_{\gamma\in \Gamma}\hat f(x-2\pi\gamma) e^{-i \ll v,x-2\pi \gamma\rr}.
\end{equation}
As $\hat f$ is rapidly decreasing on $U$, the series of the second member is absolutely convergent and defines a smooth function of $v$.

Let $b\in \CS_r^{\y}$, then $b$ is   a generalized function on $V$  with compact support (a derivative of the piecewise analytic function $b_r(\Phi,\y)$).
 Thus $b$ can be evaluated  on $V_{reg,r}$.
Let $v\in V_{reg,r}$.  The point $\lambda+v$ is again regular, and we can form
$$\sum_{\lambda\in \Lambda}  b(\lambda+v)e^{i\ll x,\lambda\rr}.$$

Consider $\hat b(x)=\int_{V}e^{i\ll x,t\rr} b(t)dt$, an analytic function on $U$. This time $\hat b$ is not rapidly decreasing, but it is a function of $x$ with at most polynomial growth.
 We may thus consider
 the series
 $$\sum_{\gamma\in \Gamma}\hat b(x-2\pi\gamma) e^{-i \ll v,x-2\pi \gamma\rr}$$
in the sense of generalized function of $v$.

We have the following theorem.

\begin{theorem}\label{theospline}
 Let $b\in \CS_r^{\y}$.
On $U\times V_{reg,r}$, we have  the following  equality of analytic functions of $(x,v)$:

 \begin{equation}\label{ourpoisson}
 \sum_{\lambda\in \Lambda} b(\lambda+v)e^{i \ll x,\lambda\rr}=\sum_{\gamma\in \Gamma}\hat b(x-2\pi\gamma) e^{-i \ll v,x-2\pi \gamma\rr}.
\end{equation}

(The second series is defined in the sense of generalized functions of $v$.)

\end{theorem}

\begin{proof}
We first prove the wanted formula for the function $b(v):=b_r(\Phi,\y)(v)$ itself.
Consider
$$F(x,v)=\sum_{\lambda\in \Lambda}  b(v+\lambda)e^{i\ll x,v+\lambda\rr}.$$

The function $F(x,v)$   is analytic in $x$ and is a function of $v$ modulo $\Lambda$.
Furthermore this function of $v$ is piecewise analytic (in $PW^{\omega}_r$), as it is a sum of a finite number of translates of $b(v)e^{i\ll x,v\rr}$. It thus defines an $L^2$ function on $V/\Lambda$.
We  form its Fourier series (in $v$)
and obtain in $L^2(V/\Lambda)$ the equality
$$F(x,v)=\sum_{\gamma\in \Gamma} a_\gamma e^{2i\pi\ll v,\gamma\rr}.$$
The coefficient $a_\gamma$ is
$\int_{v\in V/\Lambda}F(x,v) e^{-2i\pi\ll v,\gamma\rr}.$
Rewriting $F$ as a sum over $\Lambda$, we obtain
$a_\gamma=\hat b(x-2i\pi \gamma)$.
Thus we obtain the wanted formula  as an equality of $L^2$ functions of $v$.
As the first member is  analytic on $V_{reg,r}$,
 we obtain our equality everywhere on $V_{reg,r}$.

We can now derivate this equality on $V_{reg,r}\times U$ with respect to constant coefficient operators.
We obtain Theorem \ref{theospline}.

\bigskip

Let us apply Theorem \ref{theospline} to obtain an equivalent formulation of
Theorem \ref{theo:deconvtrans}. We follow the notations of Theorem \ref{theo:deconvtrans}.

Let $v\in V_{reg,r}$.
We  compute
$$W([q],v,x,\y)= \sum_{\lambda\in \Lambda} (\sum_{s\in \CV} s^{-\lambda}m_s([q],\r,\y)(\lambda+v))e^{i\ll \lambda,x\rr}.$$

Theorem \ref{theo:deconvtrans} is equivalent to the fact that for
$r$ in the zonotope, any $\epsilon$ generic in the cone $C(\Phi,r)$,  $\lim_{t>0,t\to 0} W([1],t\epsilon,x,\y)$ is identically equal to $1$.

For each $s\in \CV$, we choose a  representative $S\in U$. We denote this set of representatives still by $\CV$. Then
$$W([q],v,x,\y)=\sum_{S\in \CV} \sum_{\lambda\in \Lambda} m_s([q],\r,\y)(\lambda+v) e^{i\ll \lambda,x-S\rr}.$$

The series $m_s([q],\r,\y)\in \CS_r^{\y}[[q]]$ and
the Fourier transform of $m_s([q],\r,\y)$ is
$F_s([q],\r,x,\y).$
We can apply Poisson formula (Theorem \ref{theospline}) for each coefficient of $q^a$ in the series $m_s([q],\r,\y)$. We obtain

$$W([q],v,x,\y)=\sum_{S\in \CV} \sum_{\gamma\in \Gamma}F_s([q],\r,x-S-2\pi \gamma,\y) e^{-i\ll v,x-S-2\pi \gamma\rr}.$$

 Now $F_s([q],\r,x-S-2\pi \gamma,\y)$ is equal to
$$q^{|\Phi(s)|}e^{i([q]-1)(\ll r,x-S-2\pi \gamma\rr+\ll \r,\y\rr)}\left(\prod_{k=1}^N\frac{e^{i(\langle \alpha_k,x-S-2\pi \gamma\rangle+y_k)}s^{\alpha_k}-1}{e^{i[q](\langle \alpha_k,x-S-2\pi \gamma\rr+y_k)}s^{\alpha_k}-1}\right)$$
$$=q^{|\Phi(s)|}e^{i([q]-1)(\ll r,x-S-2\pi \gamma\rr+\ll \r,\y\rr)}
\left(\prod_{k=1}^N\frac{e^{i(\langle \alpha_k,x\rangle+y_k)}-1}{e^{i[q](\langle \alpha_k,x-S-2\pi\gamma\rr+y_k)}s^{\alpha_k}-1}\right)$$
$$=Z_s([q],\r, S+2\pi \gamma)(x,\y).$$
Here we used the fact that $e^{-i\langle \alpha_k,S+2\pi\gamma\rr}=s^{-\alpha_k}$.

 Now, the set $\{S+2\pi\gamma,\gamma\in \Gamma, S\in \mathcal V\}$ is exactly the set $\hat \CV$.
 We obtain for $v$ regular,
$$W([q],v,x,\y)=\sum_{w\in \hat \CV}Z([q],\r,w)(\y,x)e^{i\ll v,w\rr} e^{-i\ll v,x\rr}=e^{-i\ll v,x\rr}\CB([q],\r)(v)(x,\y).$$
Thus Theorem \ref{delicate}  implies Theorem \ref{theo:deconvtrans}.
\end{proof}

\section{A delicate formula}

Let $\beta\in V$.  Denote
 $$\theta(\beta)(x,y)= \frac{1}{e^{i(\ll\beta,x\rr+y)}-1},$$
 a meromorphic funtion of $(x,y)\in U_\C\times \C$.
For our proof by induction, we will use the following formula
which can be immediately verified.

Let $\beta_1,\ldots,\beta_p$ be elements of $V$ such that $\sum_{i=1}^p \beta_i=0$.
Then, we have the identity of meromorphic functions of $(x,y_1,\ldots,y_p)\in U_\C\times \C^p$,
\begin{equation}\label{eqobvious}
(e^{\sum_{j=1}^py_j}-1)\prod_{k=1}^p \theta(\beta_k)(x,y_k)
=\sum_{h=1}^p (-1)^{h-1}\prod_{1\leq j<h}\theta(-\beta_j)(x,-y_j) \prod_{h<j\leq p}
\theta(\beta_j)(x,y_j).
\end{equation}

We now prove Theorem \ref{delicate}, or rather Theorem \ref{theo;thetalim}, by induction on the number of elements in $\Phi$.
As we will need to use several  systems $\Psi$, we denote now the function $\Theta(\r,x,\y)$ by
 $$\Theta(\Phi,\r, x,\y)=\frac{e^{i(\ll r,x\rr+\ll \r,\y\rr)}}{\prod_{k=1}^N (e^{i(\langle \alpha_k,x \rangle+y_k)}-1)}.$$

Similarly, in all other objects depending on $\Phi$ introduced before, we will add the notation $\Phi$.

We first prove some technical lemma on growth, that we needed to define  the series of generalized functions  $\CB(\Phi,[q],\r)(v)$.

Let $Z(\Phi,[q],\r, w)=\sum_{a=0}^{\infty}q^a z_a(w,x,\y)$ be the Taylor expansion at $q=0$  of
 $Z(\Phi,q,\r, w)$.
Let us show that the functions $z_a(w,x,\y)$ are holomorphic in $w,x,\y$, and that the growth in $w$  is at most polynomial when $w$ varies in a coset of $2\pi \Gamma$ and $(x,\y)$ varies in a compact set of $U_\C\times \C^N$.

\begin{lemma}\label{lemtechnical}
Let $c\in U$ and  $W_c=c+2\pi \Gamma$ be a coset of $2\pi \Gamma$.
Then  for $(x,\y)$ varying in a compact subset of $U_\C\times \C^N$, the function $z_a(w,x,\y)$ is of at most polynomial growth in $w\in W_c$.
\end{lemma}

\begin{proof}

Write $w=c+2\pi n$ which $n\in \Gamma$.
Then $u_k=e^{i\ll\alpha_k,w\rr}=e^{i\ll\alpha_k,c\rr}$ does not depend on $w$ when $w$ varies in $W_c$.

We write $$Z(\Phi,q,\r,w)(x,\y)=e^{i (q-1)(\ll r,x-w\rr+\ll \r,\y\rr)}\times$$
$$\prod_{k,e^{i^{\ll \alpha_k,w\rr}}\neq 1} \frac{e^{i(\langle \alpha_k,x\rangle +y_k)}-1}
{e^{iq(\langle \alpha_k,x-w\rangle +y_k)}e^{i\langle \alpha_k,w\rangle }-1}
\prod_{k,e^{i^{\ll \alpha_k,w\rr}}= 1}
\frac{qe^{i(\langle \alpha_k,x\rangle +y_k)}-1}
{e^{iq(\langle \alpha_k,x-w\rangle +y_k)}-1} .
$$

We analyze the dependance in $w$ of each factor.
 The expansion
 $e^{i ([q]-1)(\ll r,x-w\rr+\ll \r,\y\rr)}$
 of the first factor
is of the form $e^{i \ll r,w\rr}\sum_{a} q^a p_a(w,x,\y)$ where $p_a(w,x,\y)$ is polynomial in $w$ and analytic in $x,\y$.
Thus its growth is polynomial in $w$,  for $(x,\y)$ varying in a compact subset of $U_\C\times \C^N$.

Consider the factor
$$f_k(q,w,x,\y)=\frac{e^{i(\langle \alpha_k,x \rangle+y_k)}-1}{ e^{iq(\langle \alpha_k,x-w \rangle+y_k)}e^{i\ll \alpha_k,w\rr}-1}$$ associated to $\alpha_k$
with $u_k= e^{i\ll \alpha_k,w\rr }\neq 1$.

We  define, if $u\neq 1$, constants $\beta(a,u)$ so that we have Taylor expansion
$$\frac{1}{e^zu-1}=\sum_{\ell=0}^{\infty} \beta(\ell,u) z^\ell/\ell!.$$

Then
 the Taylor series of $f_k(q,w,x,\y)$ is $$(e^{i(\langle \alpha_k,x \rangle+y_k)}-1)\sum_{\ell=0}^{\infty}\beta(\ell,u_k)
(i(\ll \alpha_k,x-w\rr +y_k))^{\ell} q^{\ell}/\ell!$$
and the coefficient of $q^\ell$ gives rise to a  function of $(x,y,w)$, polynomial in $w$ and analytic in $(x,\y)$.

If  $e^{i\ll \alpha_k,w\rr }=1$, then we write
the factor as
$$\left(\frac{ e^{i(\langle \alpha_k,x-w \rangle+y_k)}-1}
{ i(\langle \alpha_k,x-w \rangle+y_k)}\right)\left(\frac{iq( \langle \alpha_k,x-w \rangle+y_k)}{e^{iq(\langle \alpha_k,x-w \rangle+y_k)}-1}\right).$$

If $E(z)=\frac{e^{iz}-1}{iz}$, the coefficient of $q^\ell$ is
$$E(\langle \alpha_k,x-w \rangle+y_k)
b(\ell)(i(\langle \alpha_k,x-w \rangle+y_k))^{\ell}/\ell!.$$
As  $E(z)=\int_{0}^1 e^{it z}dt $, we can bound uniformly
$E(\langle \alpha_k,x-w \rangle+y_k)$ when $(x,\y)$ varies in a compact subset of $U_\C\times \C^N$ and $w$ varies in  $U$. The other terms are polynomials in $(x,\y,w)$.
\end{proof}

Thus if $W$ is a finite union of cosets of $2\pi \Gamma$, we can consider the series
$$\sum_{w\in W} Z(\Phi,[q],\r, w)(x,\y) e^{i \ll v,w\rr}.$$

This is a series of generalized function of $v$, depending holomorphically on $(x,\y)$.

\begin{lemma}\label{lemtechnicalbis}
Let $c\in U$ and  $W_c=c+2\pi \Gamma$ be a coset of $2\pi \Gamma$.
If $c\notin \hat \CV(\Phi)$,
the generalized function
$$v\mapsto \sum_{w\in W_c} Z(\Phi,[q],\r,w) e^{i \ll v,w\rr}$$
vanishes on $V_{reg,r}(\Phi)$.
\end{lemma}

\begin{proof}
Write $w=c+2\pi n$, with $n\in \Gamma$.
If we look back to the proof of Lemma \ref{lemtechnical}, we obtain that
 $z_a(w,x,\y)$ is equal to a function of $f(n,x,\y)$ polynomial in $n$ and holomorphic in $(x,\y)$,
multiplied by  $$e^{i\ll r,w\rr}\prod_{k, e^{i \ll \alpha_k,c\rr}=1}E(\langle \alpha_k,x-w \rangle+y_k).$$ Our assumption is that the corresponding $\alpha_k$  span a proper subspace $V_0$ of $V$. This subspace is
 contained in a hyperplane generated by elements of $\Phi$.
 Write $\Gamma_0=\Gamma \cap V_0^{\perp}$. We perform the summation over $W$ of $z_a(w,x,\y)e^{i\ll v,w\rr}$ by summing on  cosets of $\Gamma_0$, then on
 $ \Gamma/\Gamma_0$. When $w=c+2\pi n_1+2\pi n_0$, the functions
$E(\langle \alpha_k,x-c-2\pi n \rangle+y_k)$ do not depend on $n_0$, thus our sum is
$\sum_{n_1\in \Gamma/\Gamma_0}\sum_{n_0\in \Gamma_0} p(n_0,n_1,x,\y) e^{i\ll v+r,c+2i\pi n_1+2i\pi n_0\rr}$, where $p(n_0,n_1,x,\y)$ is polynomial in $n_0$.
Thus the corresponding generalized function of $v$  is supported on  $V_0+\Lambda-r$, which is contained on translated affine walls.
\end{proof}

\begin{corollary}\label{corvanish}
Let $W$ be a finite union of cosets of $2\pi \Gamma$ containing $\hat \CV(\Phi)$.
Then on $V_{reg,r}(\Phi)$, we have
$$\CB(\Phi, [q],\r)(v)=\sum_{w\in W} Z(\Phi,[q],\r,w) e^{i \ll v,w\rr}.$$
\end{corollary}

\begin{proof}
By definition (Definition \ref{defCB}), $\CB(\Phi, [q],\r)(v)$ is the sum over $\hat \CV(\Phi)$.
By Lemma \ref{lemtechnicalbis}, the restriction to $V_{reg,r}(\Phi)$ of the sum over the other cosets in $W\setminus {\hat \CV}(\Phi)$ vanishes on $V_{reg,r}(\Phi)$.
\end{proof}

\bigskip

We now prove Theorem \ref{theo;thetalim}.

Recall that
 $$Q(\Phi, q,\r,w,x,\y)=q^{|\Phi(\overline{w})|} \frac{e^{iq(\ll r,x-w\rr+\ll \r,\y\rr)}e^{i\ll r,w\rr}}
{\prod_{k=1}^N(e^{iq(\langle \alpha_k,x-w\rangle +y_k)}e^{i\langle \alpha_k,w\rangle}-1)}$$$$= q^{|\Phi(\overline{w})|}\Theta(\Phi,\r, qx+(1-q)w,q\y).
$$

We compute $\CT(\Phi,[q])(v)(x,\y)=\sum_{w\in \hat \CV(\Phi)} Q(\Phi,[q],w,x,\y)e^{i\ll v,w\rr}$.

First, it is equivalent to prove Theorem
 \ref{theo;thetalim} for
 $\Phi=[\alpha_1,\alpha_2,\ldots, \alpha_N]$ or for
 $\Phi_u=[u_1\alpha_1,u_2\alpha_2,\ldots, u_N\alpha_N]$
 with $u_i=\pm 1$.

For example, let $\Phi'=[-\alpha_1,\alpha_2,\ldots, \alpha_N]$.
If $r\in Z(\Phi)$, then $r'=r-\alpha_1\in Z(\Phi')$, and
$Cone(r,\Phi)=Cone(r',\Phi')$. Indeed if $r+t\epsilon\in Z(\Phi)$ for $t$ small, then $r-\alpha_1+t\epsilon\in Z(\Phi)-\alpha_1=Z(\Phi')$.
 The sets $\hat \CV(\Phi)$ and $\hat \CV(\Phi')$ are equal.  The sets  $V_{reg,r}(\Phi)$ and $V_{reg,r'}(\Phi')$ are equal.

Consider a sequence  $\r$  of  $N$ real numbers, which is a $\Phi$-representation of $r$,  that is
$r=r_1\alpha_1+\cdots+ r_N\alpha_N$.
Then $\r'=[1-r_1,r_2,\ldots, r_N]$
 is a $\Phi'$-representation of  $r'=r-\alpha_1$.
Let $\y=[y_1,y_2,\ldots, y_N]$ and $\y'=[-y_1,y_2,\ldots, y_N]$.
Then $\ll \r,\y\rr= \ll \r',\y'\rr+y_1$.

Using the relation
$e^{z}\frac{1}{e^z-1}=-\frac{1}{e^{-z}-1}$,
we see that
$$\Theta(\Phi,\r,x,\y)=-\Theta(\Phi',\r',x,\y')$$
and consequently
$$Q(\Phi,q,\r,w,x,\y)=-Q(\Phi',q,\r',w,x,\y'),$$
$$\CT(\Phi,[q],\r)(v)(x,\y)=-\CT(\Phi',[q],\r')(v)(x,\y').$$

Thus, if Theorem \ref{theo;thetalim} is true for $\Phi'$,  we obtain  if $r\in Z(\Phi)$ and $\epsilon\in Cone(r,\Phi)$,
$$\lim_{t>0, t\to 0}\CT(\Phi,q,\r)(t\epsilon)(x,\y)|_{q=1}=
- \Theta(\Phi',\r',x,\y')=\Theta(\Phi,\r,x,\y).$$

\bigskip

Let ${\bf M}=[M_1,M_2,\ldots, M_N]$ be a sequence of positive integers, and let  $\Phi_{\bf M}=[M_1\alpha_1,M_2\alpha_2,\ldots, M_N\alpha_N].$

Similarly, let us see that if Theorem \ref{theo;thetalim} is true for
$\Phi_{\bf M}$, then Theorem \ref{theo;thetalim}  is true for $\Phi$.
For example, let
$$\Phi_M=[M\alpha_1,\alpha_2,\ldots, \alpha_N].$$
Let
$\y_M=[My_1,y_2,\ldots, y_N]$.
If $\r=[r_1,r_2,\ldots, r_N]$ is a  $\Phi$-representation of $r$ for $\Phi$, then
 $\r_d=[(r_1+d)/M,r_2,\ldots, r_N]$
is a $\Phi_M$-representation of $r+d\alpha_1$.
We use $$\frac{1}{e^{z}-1}=\frac{1+e^{z}+\cdots+e^{(M-1)z}}{e^{Mz}-1}.$$
Then
\begin{equation}\label{eqtrue}
\Theta(\Phi,\r,x,\y)=\sum_{d=0}^{M-1}
\Theta(\Phi_M,\r_d,x,\y_M).
\end{equation}

From Equation (\ref{eqtrue}),
$$Q(\Phi,q,\r,w,x,\y)=\sum_{d=0}^{M-1}
Q(\Phi_M,q,\r_d,w,x,\y_M).$$

Assume that $r\in Z(\Phi)$, $\epsilon\in Cone(r,\Phi)$.
Then for  $d=0,\ldots,M-1$, $r+d\alpha_1\in Z(\Phi_M)$ and $\epsilon\in Cone(r+d\alpha_1,\Phi_M)$.
By Corollary \ref{corvanish},
we can compute  the series $\CT(\Phi,[q],\r)$ by summing over the set $\hat \CV(\Phi_M)$ which contains $\hat \CV(\Phi)$.
Then we obtain on $V_{reg,r}(\Phi)=V_{reg,r}(\Phi_M)$
$$\CT(\Phi,[q],\r)(v)(x,\y)=
\sum_{d=0}^{M-1}
\CT(\Phi_M,[q],\r_d)(v)(x,\y_M).$$

Taking limits and using Equation (\ref{eqtrue}) , we obtain Theorem \ref{theo;thetalim}  for $\Phi$, if we have proven Theorem \ref{theo;thetalim} for $\Phi_M$.

\bigskip

We are now  ready to proceed on our induction.
If the number $N$ of elements of $\Phi$ is equal to the dimension of $V$,
then $\alpha_1,\alpha_2,\ldots, \alpha_N$ form a basis of $V$.
We may take as lattice containing the elements $\alpha_k$, the lattice with basis $\alpha_k$, and
we are reduced to the calculation in dimension $1$ that we have already done  at the end of  Section \ref{secillumin}.

If $N>\dim(V)$,
let us consider a relation
between elements of $\Phi$. We may assume after eventual  relabeling and  changing signs  that the relation is of the form
$\sum_{k=1}^p M_k\alpha_k=0$, with $M_k$ positive number.

It is sufficient to prove Theorem \ref{theo;thetalim}
for $\Phi_{\bf M}=[M_1\alpha_1,\ldots, M_p \alpha_p,\alpha_{p+1},\ldots, \alpha_N]$.
  Renaming the list, we are reduced  to prove
Theorem \ref{theo;thetalim} for a system $\Psi=[\beta_1,\beta_2,\ldots, \beta_N]$ with
a relation
$\sum_{k=1}^p \beta_k=0$.

We will prove the identity of  Theorem \ref{delicate} for $\Psi$, when $\y$  is small and outside the
hyperplane $\sum_{k=1}^py_k=0$ in $\C^N$.
As, after multiplying by $\prod_{k=1}^N (e^{i(\ll\beta_k,x\rr+y_k)}-1)$,
the identity to be proven is analytic in $(x,\y)$, this will be sufficient.

We assume $r$ in the zonotope $Z(\Psi)$ and
we choose a representation $r=\sum_{k=1}^N r_k \beta_k$ with $0\leq r_k\leq 1$. Similarly we can represent $\epsilon\in Cone(r,\Phi)$ as
$\epsilon=\sum_{k=1}^N s_k \beta_k$ with $s_k\geq 0$, if $r_k=0$, and $s_k\leq 0$ if $r_k=1$.
In this case the curve $r+t \epsilon$ stays in $Z(\Phi)$ when $t>0$ and small.

We relabel eventually the first $p$ elements of $\Psi$ so that  the sequence $[r_1+ts_1,\ldots, r_p+ts_p]$ is weakly increasing when $t$ is small and positive.
That is $$r_1\leq  r_2\leq \cdots \leq  r_p$$
and if $r_{u}=r_{u+1}$, we take an order so that $s_{u}\leq s_{u+1}$.

Define
$$\Psi_1=[\beta_2,\ldots, \beta_p,\beta_{p+1},\ldots, \beta_{N}]$$
$$\Psi_2=[-\beta_1,\beta_3,\ldots, \beta_p,\beta_{p+1},\ldots, \beta_{N}]$$
$$\cdots$$
$$\Psi_p=[-\beta_1,-\beta_2,\ldots, -\beta_{p-1},\beta_{p+1},\ldots,\beta_N].$$
Systems $\Psi_i$ have $N-1$ elements.

Define for $1\leq h\leq p$,
$$\r_h=[(r_h-r_1),\ldots, (r_{h}-r_{h-1}),(r_{h+1}-r_h),\ldots,(r_p-r_h),r_{p+1},\ldots, r_N],$$
$$\y_h=[-y_1,\ldots, -y_{h-1},y_{h+1},\ldots,y_{p}, y_{p+1},\ldots, y_{N}]$$

Then $\r_h$ is a $\Psi_h$-representation of $r$.

The following is the crucial proposition. It is taken from (\cite{SZ}, Lemma 1.8).
\begin{proposition}

Let $r\in Z(\Psi)$ and $\epsilon\in Cone(r,\Psi)$.
Then

$\bullet$
The vector $\epsilon$ is generic for $\Psi_h$,
$r\in Z(\Psi_h)$ and
$\epsilon\in Cone(r,\Psi_h)$ for $1\leq h\leq p$.

$\bullet$

\begin{equation}\label{eqcrucial}
\Theta(\Psi,\r, x,\y)=
\sum_{h=1}^pc_h(\r,\y)\Theta(\Psi_h,\r_h)(x,\y_h)
\end{equation}
with
$c_h(\y,\r)=(-1)^{h+1}\frac{e^{ir_h(\sum_{j=1, j\neq h}^p y_j)}}{e^{i(\sum_{j=1}^p y_j)}-1}.$

\end{proposition}

\begin{proof}

It is clear that if $\epsilon$ is generic for $\Psi$, it is generic for the smaller system $\Psi_h$.

We have
$r=\sum_{k=1,k\neq h}^p (r_h-r_{k})(-\beta_k)+
\sum_{k=h+1}^p (r_k-r_{h})(\beta_k)$.
We have
$0\leq (r_h-r_{k})\leq 1$ for $k<h$, and similarly
 $0\leq (r_k-r_h)\leq 1$ for $h<k\leq p$.
 Thus $r$ belongs to the zonotope $Z(\Psi_h)$.
 Similarly, our choice of order implies that  the curve $r(t)=r+t \epsilon$ stays in $Z(\Psi_h)$.

Equality (\ref{eqcrucial}) follows from
Equality(\ref{eqobvious}) which is
Equality (\ref{eqcrucial}) for $\r=0$.
We just multiply  by $e^{i(\ll r,x\rr+\ll \r,\y\rr)}$ and remark that $$\sum_{j\neq h}y_j(r_j-r_h)=\ll \y,\r\rr -r_h \sum_{j\neq h}y_j.$$

\end{proof}

Equation (\ref{eqcrucial}) implies that
$$Q(\Psi,q,\r,w, x,\y)=
\sum_{h=1}^pc_h(\r,q\y)Q(\Psi_h,q,\r_h,w,x,\y_h).$$

Remark that the functions $\y\to c_h(\r,q\y)$ are defined if $|q|<2$ and $\y$ sufficiently small and outside the hyperplane
$\sum_{k=1}^p y_k=0.$

Taking Taylor series, we obtain
$$Q(\Psi,[q],\r,w,x,\y)=
\sum_{h=1}^{p+q} c_h(\r,[q]\y)Q(\Phi_h,[q],\r_h,w,x,\y_h).$$

We sum over the set $\hat \CV(\Psi)$ which contains the sets
$\hat \CV(\Psi_h)$. So we obtain over $V_{reg,r}(\Psi)$ (contained in $V_{reg,r}(\Psi_h)$):

$$\CT(\Psi,[q],\r)(v)(x,\y)=
\sum_{h=1}^{p} c_h(\r,[q]\y)\CT(\Psi_h,[q],\r_h)(v)(x,\y_h).$$

So for $\epsilon$ generic, we obtain
$$\lim_{t>0,t\to 0}\CT(\Psi,[q],\r)(t\epsilon)(x,\y)=
\sum_{h=1}^{p} c_h(\r,[q]\y)\lim_{t>0,t\to 0}\CT(\Psi_h,[q],\r_h)(t\epsilon)(x,\y_h).$$
 When $\y$ is sufficiently small, the Taylor series  of
$c_h(\r,[q]\y)$ converges for $q=1$ to  $c_h(\r,\y)$.
By induction hypothesis, if $r\in Z(\Psi)$ and $\epsilon\in Cone(r,\Psi)$, and our crucial proposition,
the limit converges for $q=1$ to
$$\sum_{h=1}^{p} c_h(\r,\y)\Theta(\Psi_h,\r_h,x,\y_h)=\Theta(\Psi,\r,x,\y).$$

This is the end of the proof of Theorem \ref{theo;thetalim}.

\part{Applications}

\section{Deconvolution formula for the Box spline with parameters}

As  we discuss in the introduction, we can apply Theorem \ref{deconvtrans} to invert the semi-discrete convolution by the Box spline with parameters.
We assume that $$\Phi=[\alpha_1,\alpha_2,\ldots,\alpha_N]$$ spans $V$  and is contained in a lattice $\Lambda$.

Let $f\in {\mathcal C}(\Lambda)$ be a function on $\Lambda$.
Then
$$P(f,\y)(v)=\sum_{\lambda\in \Lambda}f(\lambda) b(\Phi,\y)(v-\lambda)$$ is a piecewise analytic function on $V$.
  In other words, $P(f,\y)(v)dv$ is the convolution of the discrete measure $\sum_{\lambda\in \Lambda} f(\lambda)\delta_\lambda^V$ with the distribution with compact support $B(\Phi,\y)$.

When $\Phi$ is unimodular, the map $f\to P(f,\y)$ is injective, and Dahmen-Micchelli deconvolution formula computes the inverse map. We now give a general deconvolution formula for the semi-discrete convolution with the Box spline with parameters.

Let $\CV$ be the vertex set for the system $\Phi$.
Let $\partial_{\alpha_k}$ be the differentiation in the direction $\alpha_k$.
Let $s\in \CV$.
 We divide the list  $\y$ in sublists
$[\y_0,\y_1]$,   $\y_0$ corresponding to the indices $k$ with  $\alpha_k\in \Phi(s)$, and $\y_1$ to the indices not in $\Phi(s)$.

Consider the  locally analytic function
$$b(\Phi,s,\y)=\prod_{k, s^{\alpha_k}\neq 1} (s^{\alpha_k}e^{iy_k}\delta_{\alpha_k}^V-1)*b(\Phi(s),\y_0)$$
a sum of translated of the Box spline of the system
$\Phi(s).$

Then define the locally analytic function $P(s,\y,f)$  by
$$P(s,\y,f)(v)=\sum_{\xi\in \Lambda} s^{\xi} f(\xi)b(\Phi,s,\y)(v-\xi).$$

We can recover the value of $f$  at the point $\lambda\in \Lambda$ from the knowledge, in the neighborhood of $\lambda$, of the functions
$P(s,\y,f)(v)$, for all $s\in \CV$.

Define the series of differential operators
$$Todd([q],s,\y)(\partial)=\prod_{k, s^{\alpha_k}=1} \frac{q(-\partial_{\alpha_k} +iy_k)}{e^{[q](-\partial_{\alpha_k} +iy_k)}-1}
\prod_{k, s^{\alpha_k}\neq 1} \frac{1}{e^{[q](-\partial_{\alpha_k} +iy_k)}s^{\alpha_k}-1}.$$

\begin{theorem}\label{theoinv}
Let $f\in \mathcal C(\Lambda)$.
Let $\y$ small.
For $\epsilon$ generic, and $\lambda\in \Lambda$, the series
$$\lim_{t>0, t\to 0}(Todd([q],s,\y)(\partial)P(s,\y,f))(\lambda+t\epsilon)$$ is convergent at $q=1$.

Furthermore, for $\epsilon$ generic {\bf in the cone generated by} $\Phi$,
\begin{equation}\label{eqf}
f(\lambda)=\sum_{s\in \CV} s^{-\lambda} \lim_{t>0, t\to 0}(Todd([q],s,\y)(\partial)P(s,\y,f))(\lambda+t\epsilon)|_{q=1}.
\end{equation}
\end{theorem}

\begin{proof}

Let us see that this is just a reformulation of Theorem \ref{deconvtrans}.
By linearity, we need to prove the formula for
$\delta$ functions at any point of the lattice $\Lambda$.

We use the following formula.
\begin{equation}\label{eq;Fandboxy}
F_s(q,x,\y)=D_s(q,x,\y)
\left(\prod_{k,s^{\alpha_k}\neq 1}(e^{i(\langle \alpha_k,x\rangle+y_k)}s^{\alpha_k}-1)\right)\left(
\prod_{k,s^{\alpha_k}=1}
\frac{e^{i(\langle \alpha_k,x\rangle+y_k)}-1}{i(\langle \alpha_k,x\rangle+y_k)}
\right)
\end{equation}
where
$$D_s(q,x,\y)=
\left(\prod_{k,s^{\alpha_k}\neq 1}\frac{1}{e^{iq(\langle \alpha_k,x\rangle+y_k)}s^{\alpha_k}-1}\right)
\left(\prod_{k,s^{\alpha_k}=1}\frac{iq(\langle \alpha_k,x\rangle+y_k)}{e^{iq(\langle \alpha_k,x\rangle+y_k)}-1}\right).$$

 If $f_0=\delta_0^\Lambda$ is the delta function at $0$ on the lattice $\Lambda$, then taking Taylor expansions and Fourier transforms, we see that
$Todd([q],s,\y)(\partial)P(s,\y,f_0)$  is the series $m_s([q],\y)$, and the theorem is equivalent to Theorem \ref{deconvtrans}.

Otherwise,  for $f=\delta_\kappa^{\Lambda}$, we see that
$P(s,\y,f)(v)=s^{\kappa}P(s,\y,f_0)(v-\kappa)$.
Thus
$$\sum_{s\in \CV} s^{-\lambda} D([q],s,\y)P(s,\y,f)(\lambda+t\epsilon)=
\sum_{s\in \CV} s^{-(\lambda-\kappa)} D([q],s,\y)P(s,\y,f_0)(\lambda-\kappa+t\epsilon).$$
By the preceding computation, this is $f_0(\lambda-\kappa)=f(\lambda)$.
\end{proof}

We can as well obtain a deconvolution formula for the translated box spline.
Let $r\in Z(\Phi)$.
Let $f\in {\mathcal C}(\Lambda)$ be a function on $\Lambda$.
Then define
$$P(f,\r,\y)(v)=\sum_{\lambda\in \Lambda}f(\lambda) b_r(\Phi,\y)(v-\lambda).$$
Then $P(f,\r,\y)\in PW^{\omega}_r$.
Let $s\in \CV$. We divide our list $\r$ in sublists
  $[\r_0,\r_1]$  corresponding to the indices in $\Phi(s)$, and not in $\Phi(s)$.
Define
$$B(\Phi,s,\r,\y)(v)=e^{-i\ll \r_1,\y_1\rr}\delta_{-r_1}*\prod_{k, s^{\alpha_k}\neq 1} (s^{\alpha_k}e^{iy_k}\delta_{\alpha_k}^V-1)*B(\Phi(s),\r_0, \y_0)$$
a sum of translates of the Box spline of the system
$\Phi(\s)$.
Remark that  $B(\Phi,s,\r,\y)$ is supported on $Z(\Phi)-r$.

Then define the locally analytic function $P(s,\r,\y,f)$  by
$$P(s,\r,\y,f)(v)=\sum_{\xi\in \Lambda} s^{\xi} f(\xi)B(\Phi,s,\r,\y)(v-\xi).$$

Define the series of differential operators
$$Todd(\Phi,[q],s,\r,\y)(\partial)=e^{[q](-\partial_{r}+i\ll \y,\r\rr)} \prod_{k, s^{\alpha_k}=1} \frac{q(-\partial_{\alpha_k} +iy_k)}{e^{[q](-\partial_{\alpha_k} +iy_k)}-1}
\prod_{k, s^{\alpha_k}\neq 1} \frac{1}{e^{[q](-\partial_{\alpha_k} +iy_k)}s^{\alpha_k}-1}.$$

Then, the following theorem is just the reformulation of Theorem \ref{theo:deconvtrans}.

\begin{theorem}
Let $r$ be in the zonotope.
Let $f\in \mathcal C(\Lambda)$.
Let $\y$ small.
For $\epsilon$ generic, and $\lambda\in \Lambda$, the series
$$\lim_{t>0, t\to 0}(Todd(\Phi,[q],s,\r,\y)(\partial)P(s,\r,\y,f))(\lambda+t\epsilon)$$ is convergent at $q=1$.

Furthermore, for $\epsilon$ generic {\bf in the cone} $Cone(r,\Phi)$,
\begin{equation}\label{eqf}
f(\lambda)=\sum_{s\in \CV} s^{-\lambda} \lim_{t>0, t\to 0}(Todd([q],s,\r,\y)(\partial)P(s,\r,\y,f))(\lambda+t\epsilon)|_{q=1}.
\end{equation}
\end{theorem}

In particular, if $r=\rho$ is the center of the zontope, we can take limits in any directions, as $C(\rho,\Phi)=V$.
This will be important to define multiplicities formulae for the Dirac operators.

We reformulate the deconvolution formulae using the local pieces of the box spline.
Let $\c$ be an alcove. Consider
$\Delta=(\c-Z(\Phi))\cap \Lambda.$

For any  $f\in {\mathcal C}(\Lambda)$, we can
reconstruct $f$ on $\Delta$, using the functions $P(s,\y,f).$
Let  $P(s,\y,f,\c)$ be the analytic function on $V$, coinciding with   $P(s,\y,f)$  on $\c$.

\begin{theorem}\label{theorecons}
Let $\c$ be an alcove and let
$\lambda\in (\c-Z(\Phi))\cap \Lambda$,
then
$$f(\lambda)=\sum_{s\in \CV} s^{-\lambda} (Todd([q],s,\y)(\partial)P(s,\y,f,\c))(\lambda)|_{q=1}.$$
\end{theorem}

\begin{proof}
Let $\lambda\in (\c-Z(\Phi))\cap \Lambda$.
We choose $r\in Z(\Phi)$ such that $\lambda+r\in \c$.
 Thus  $\lambda\in V_{reg,r}$ and belongs to the translated alcove $\c-r$. Thus
 function $P(s,\y,\r,f)$ near $\lambda$   is just $P(s,\y,f,\c)(v+r)$.
 The deconvolution formula for the translated Box spline asserts that
$$f(\lambda)=
\sum_{s\in \CV} s^{-\lambda} (Todd([q],s,\r,\y)(\partial)P(s,\y,\r,f))(\lambda)|_{q=1}.$$
As $(Todd([q],s,\r,\y)(\partial)=e^{\ll \r,\y\rr}e^{-[q]\partial_r}(Todd([q],s,\y)(\partial)$, we obtain our formula.

\end{proof}

This theorem shows that if $\tau$ is a union of alcoves $\c_i$, such that for any $s\in \hat\CV$ the analytic function  $P(s,\y,f,\c_i)$  coincides, we obtain a reconstruction formula for $f$ on $\tau-Z(\Phi)$.
In the next section, we apply this to Kostant partition function with parameters.

We can use this theorem for $\y=0$ and $f_0=\delta_0^{\Lambda}$.
We use notations of \cite{dpv1}.
The Dahmen-Micchelli space $DM(\Phi)$ is a space of $\Z$-valued functions on $\Lambda$ satisfying some difference equations.
Notations and definitions are as in \cite{dpv1}.

It is  possible to define a space $DM(\Phi,\y)$
 with value in the ring $\Z[e^{iy_1},e^{iy_2},\cdots, e^{iy_N}]$, consisting of the functions $f$   satisfying the equation
$\prod_{\alpha_k\in C} (\nabla_k-e^{iy_k})f=0$ for all cocircuits
and to compute  the structure of $DM(\Phi,\y)$.
However, we do not undertake this task for the moment.
 We just take $\y=0$, and relate our theorem  on translated Box splines to results of \cite{dpv1}.

We return to the notations of Subsection \ref{subinv}.
We denote our series $m_s([q],0)$ simply by $m_s([q])$.
Let $\c$ an alcove contained in $Z(\Phi)$.
Consider the polynomial function $m_s(\c)$ on  $V$ such that $m_s([1])$ coincide with $m_s(\c)$ on $\c$. It is a function in the Dahmen-Micchelli space  of polynomials $D(\Phi(s))$.
The restriction of $m_s(\c)$ to $\Lambda$ is a polynomial function on $\Lambda$.
Define $Q(\c)$ to be the quasi polynomial on $\Lambda$:
$$Q(\c)=\sum_{s\in \CV}{\hat s}^{-1}m_s(\c)|_\Lambda.$$
Then $Q(\c)$ belongs to $DM(\Phi)$.
Theorem  \ref{theorecons} for $f=\delta_0^{\Lambda}$  and $\y=0$ gives the following result, proved in \cite{dpv1}.

\begin{theorem}\label{theosameasdpv}
$Q(\c)$ is the unique Dahmen-Micchelli quasipolynomial such that
$Q(\c)(\nu)=1$ if $\nu=0$, and
$Q(\c)(\nu)=0$ if $\nu\in (\c-Z(\Phi))\cap \Lambda.$
\end{theorem}

\section{Kostant Partition functions with parameters}
Let $\Phi$ be a series of non zero elements of a lattice $\Lambda\subset V$, and assume that $\Phi$ generates a salient cone.

Consider the group $T$ with character group $\Lambda$.
Let $M_\Phi=\oplus_{k=1}^N L_{\alpha_k}$ the linear representation space of $T$.
Here $L_{\alpha_k}=\C e_{\alpha_k}$, and   $t e_{\alpha_k}=t^{\alpha_k} e_{\alpha_k}$.
Consider the action of the diagonal group  $D_N=\{e^{i\y}=[e^{i y_1},e^{iy_2},\ldots,e^{iy_N}]\}$ on $M_\Phi$ acting by $e^{iy_k}$ on $L_{\alpha_k}$.
Consider the symmetric algebra  $S(M_\Phi)$ of $M_\Phi$. It decomposes as a $T$-module as
$$S(M_\Phi)=\oplus_{\nu\in \Lambda }S_\nu.$$
The space $S_\nu$ is finite dimensional.
The group $D_N$ acts on $S_\nu$, and

$$f(\y)(\nu)= {\rm Trace}_{S_\nu}(e^{i\y})$$ is a function on $\Lambda$.

If $\y=0$, then $f(0)(\nu)=\dim S_\nu$ is the value of the partition function at $\nu$, that is the
 cardinal of the set $P(\Phi,\nu)$  of sequence $\p=[p_1,p_2,\ldots, p_N]$ of non negative integers $p_k$ such that $\sum_{k=1}^N p_k \alpha_k=\nu$:
$$P(\Phi,\nu)=\{\p\geq 0,\sum_{k=1}^N p_k \alpha_k=\nu\}.$$

For any $\y$, we have the formula
$${\rm Trace}_{S_\nu}(e^{i\y})\sum_{\p\in P(\Phi,\nu)}e^{i p_ky_k}.$$

 Now consider the following multispline distributions $T(\Phi,\y)$  on $V$ such that, for a  continuous function  $F$ on $V$,
$$
\langle T(\Phi,\y),F\rangle=
\int_{0}^{\infty}\cdots \int_{0}^{\infty} e^{i(\sum_{k=1}^N t_k y_k)}
F(\sum_{k=1}^{N} t_k\alpha_k) dt_1\cdots dt_N.
$$

Consider the  cones generated by subsets of $\Phi$, and  consider $V_{reg,plus}$ as the complement of the union of the boundaries of these cones. A chamber $\tau$ is defined as a connected component of $V_{reg,plus}$.
Then it is easy to see that for each chamber $\tau$, there exists an analytic function  $T^{\tau}(\Phi,\y)(v) $ of $(v,\y)$ so that
 $T(\Phi,\y)(v)=T^{\tau}(\Phi,\y)(v)$ when $v\in \tau$.

Writing the quadrant $\R_+^N$ as the union of the translates of the hypercube $\{t_k, 0\leq t_k\leq 1\}$,
we can write $T(\Phi,\y)$ as a sum of translated of $B(\Phi,\y)$.

\begin{equation}\label{eq;partition}
T(\Phi,\y)=\sum_{p_1=0}^{\infty} \cdots \sum_{p_N=0}^{\infty} e^{i(\sum_{k=1}^N p_k y_k)} B(\Phi,\y)(v-\sum_{k=1}^N p_k \alpha_k).
\end{equation}

Consider the zonotope $Z(\Phi)=\sum_{k=1}^N [0,1]\alpha_k$   generated by $\Phi$.

The following formula was proved in Brion-Vergne, via cone decompositions. We gave also a proof in Szenes-Vergne, via residues. Here is yet another proof.

\begin{theorem}
Let $\tau$ be a chamber, and $\y$ small.
Then
for $\nu\in (\tau-Z(\Phi))\cap \Lambda$,
$${\rm Trace}_{S(\nu)}e^{i\y}=\sum_{s\in \CV} s^{-\nu} (Todd([q],\y, s)(\partial) T^{\tau}(\Phi,\y))(\nu)|_{q=1}.$$
\end{theorem}
Thus the function
${\rm Trace}_{S(\nu)}e^{i\y}$ is given by an analytic function of $\nu,\y$ on each enlarged chamber $\tau-Z(\Phi)$.

\begin{proof}
This is a direct consequence of Theorem \ref{theoinv}.
Indeed, we just need to verify that the functions
$P(s,f,\y)$ coincide with an analytic function on $\tau$.
But we see that
$P(s,f,\y)$ is just equal to $T(\Phi(s),\y_0)$. The chamber $\tau$ for $\Phi$ is smaller than the chamber  for $\Phi(s)$ conaining $\tau$, so
$T(\Phi(s),\y_0)$ is analytic on $\tau$.
\end{proof}

\end{document}